\documentclass[12pt]{article}

\usepackage{graphicx}
\usepackage[dvips]{epsfig}
\usepackage{amsmath}
\usepackage{amsthm}
\usepackage{amssymb}
\usepackage[mathscr]{eucal}
\usepackage[all]{xy}
\usepackage{exscale,relsize}
\usepackage{psfrag}
\usepackage{bbm}

\title{\bf Rouquier Complexes are Functorial over Braid Cobordisms}
\author{Ben Elias  and  Daniel Krasner}
\date{}

\theoremstyle{plain}
\newtheorem{theorem}{Theorem}

\newtheorem{lemma}[theorem]{Lemma}
\newtheorem{cor}[theorem]{Corollary}
\newtheorem{claim}[theorem]{Claim}

\newtheorem{question}[theorem]{Question}

\newtheorem{notation}[theorem]{Notation}
\newtheorem{defn}[theorem]{Definition}

\def\proof{{\bf {\medskip}{\noindent}Proof. }}

\def\note{{\bf {\bigskip}{\noindent}Note: }}

\def\title{\em}

\usepackage[top=1in, bottom=1in, left=1in, right=1in]{geometry}

\renewcommand{\xi}{x_i}

\newcommand{\TenR}{\otimes}

\newcommand{\ii}{\underline{\textbf{\textit{i}}}}
\newcommand{\jj}{\underline{\textbf{\textit{j}}}}

\newcommand{\Hom}{{\rm Hom}}

\newcommand{\HOM}{{\rm HOM}}
\renewcommand{\to}{\rightarrow}

\newcommand{\id}{{\rm id}}
\newcommand{\define}{\stackrel{\mbox{\scriptsize{def}}}{=}}

\newcommand{\Ztt}{\ensuremath{\mathbb{Z}\left[t,t^{-1}\right]}}

\newcommand{\mc}[1]{\mathcal{#1}}

\newcommand{\ig}[2]{\vcenter{\xy (0,0)*{\includegraphics[scale=#1]{#2}} \endxy}}

\newcommand{\igc}[2]{\begin{center} \includegraphics[scale=#1]{#2} \end{center}}

\newcommand{\auptob}[2]{\xy  (0,-5)*+{#1}="1"; (0,5)*+{#2}="2"; {\ar@{|->} "1";"2"};\endxy}

\newcommand{\superscript}[1]{\ensuremath{^{\textrm{#1}}}}
\renewcommand{\th}{\superscript{th}}
\newcommand{\st}{\superscript{st}}

\def\R{{\mathbb R}}
\def\Z{{\mathbb Z}}

\def\1{{\mathbbm 1}}

\begin{document}

\baselineskip 14pt
\maketitle
 
\vspace{0.1in} 

\begin{abstract} Using the diagrammatic calculus for Soergel bimodules developed by B. Elias and M. Khovanov, we show that Rouquier complexes are functorial
over braid cobordisms. We explicitly describe the chain maps which correspond to movie move generators. \end{abstract}

\vspace{0.1in}

\footnotetext[1]{The authors we partially supported by NSF grants DMS-0739392 and DMS-0706924.}

\tableofcontents

\pagebreak

\section{Introduction}
\label{sec-introduction}

For some time, the category of Soergel bimodules, here called $\mc{SC}$, has played a significant role in the study of representation theory, while more recently strong connections between $\mc{SC}$ and knot theory have come to light. Originally introduced by Soergel in \cite{Soe1}, $\mc{SC}$ is an equivalent
but more combinatorial description of a certain category of Harish-Chandra modules over a semisimple lie algebra $\mathfrak{g}$. The added
simplicity of this formulation comes from the fact that $\mc{SC}$ is just a full monoidal subcategory of graded $R$-bimodules, where $R$ is a
polynomial ring equipped with an action of the Weyl group of $\mathfrak{g}$. Among other things, Soergel gave an isomorphism between the
Grothendieck ring of $\mc{SC}$ and the Hecke algebra $\mc{H}$ associated to $\mathfrak{g}$, where the Kazhdan-Lusztig generators $b_i$ of $\mc{H}$
lift to bimodules $B_i$ which are easily described. The full subcategory generated monoidally by these bimodules $B_i$ is here called $\mc{SC}_1$,
and the category including all grading shifts and direct sums of objects in $\mc{SC}_1$ is called $\mc{SC}_2$. It then turns out that $\mc{SC}$ is
actually the idempotent closure of $\mc{SC}_2$, which reduces the study of $\mc{SC}$ to the study of these elementary bimodules $B_i$ and their tensors.
For more on Soergel bimodules and their applications to representation theory see \cite{Soe2, Soe3, Soe4}.

An important application of Soergel bimodules was discovered by Rouquier in \cite{Rou1}, where he observes that one can construct complexes in $\mc{SC}_2$
which satisfy the braid relations modulo homotopy. To the $i\th$ overcrossing (resp. undercrossing) in the braid group Rouquier associates a
complex, which has $R$ in homological degree 0 and $B_i$ in homological degree $-1$ (resp. $1$). Giving the homotopy equivalence classes of
invertible complexes in $\mc{SC}_2$ the obvious group structure under tensor product, this assignment extends to a homomorphism from the braid
group. Using this, one can define an action of the braid group on the homotopy category of $\mc{SC}_2$, where the endofunctor associated to a crossing
is precisely taking the tensor product with its associated complex.

Following his work with L. Rozansky on matrix factorizations and link homology in \cite{KR}, Khovanov produced an equivalent categorification \cite{K1} of the
HOMFLY-PT polynomial utilizing Rouquier's work. To a braid one associates its Rouquier complex, which naturally has two gradings: the homological grading, and
the internal grading of Soergel bimodules. Then, taking the Hochschild homology of each term in the complex, one gets a complex which is triply graded (the
third grading is the Hochschild homological grading). Khovanov showed that, up to degree shifts, this construction yields an equivalent triply-graded complex
to the one produced by the reduced version of the Khovanov-Rozansky HOMFLY-PT link homology for the closure of the braid (see \cite{K1} and \cite{KR} for more
details).

Many computations of HOMFLY-PT link homology were done by B. Webster \cite{W2}, and by J. Rasmussen in \cite{Ras1} and \cite{Ras2}. In the latter paper
\cite{Ras2}, Rasmussen showed that given a braid presentation of a link, for every $n \in \mathbb{N}$ there exists a spectral sequence with $E^1$-term its
HOMFLY-PT homology and the $E^{\infty}$-term its $sl(n)$ homology. This was a spectacular development in understanding the structural properties of these
theories, and has also proven very useful in computation (see for example \cite{MV}).

One key aspect of the original Khovanov-Rozansky theory is that it gives rise to a projective functor. The braid group can actually be realized as the
isomorphism classes of objects in the category of braid cobordisms. This category, while having a topological definition, is equivalent to a combinatorially
defined category, whose objects are braid diagrams, and whose morphisms are called \emph{movies} (see Carter-Saito, \cite{CS}). For instance, performing a
Reidemeister 3 move on a braid diagram would give an equivalent element of the braid group, but gives a distinct object in the braid cobordism category;
however, the R3 move itself is a movie which gives the isomorphism between those two objects. It was shown in \cite{KR} that for each movie between braids one
can associate a chain map between their triply-graded complexes. This assignment was known to be \emph{projectively functorial}, meaning that the relations
satisfied amongst movies in the braid cobordism category are also satisfied by their associated chain maps, up to multiplication by a scalar. Scalars take
their value in $\mathbb{Q}$, the ring over which Khovanov-Rozansky theory is defined. However, these chain maps are not explicitly described even in the
setting of Khovanov-Rozansky theory, and the maps they correspond to in the Soergel bimodule context are even more obscure. A more general discussion of
braid group actions, including this categorification via Rouquier complexes, and their extensions to projective actions on the category of braid cobordisms
can be found in \cite{KT}.

Recently, in \cite{EKh}, the first author in conjunction with Mikhail Khovanov gave a presentation of the category $\mc{SC}_1$ in terms of generators and
relations. Moreover, it was shown that the entire category can be drawn graphically, thanks to the biadjointness and cyclicity properties that the category
possesses. Each $B_i$ is assigned a color, and a tensor product is assigned a sequence of colors. Morphisms between tensor products can be drawn as certain
colored graphs in the plane, whose boundaries on bottom and top are the sequence of colors associated to the source and target. Composition and tensor product
of morphisms correspond to vertical and horizontal concatenation, respectively. Morphisms are invariant under isotopy of the graph embedding, and satisfy a
number of other relations, as described herein. In addition to providing a presentation, this graphical description is useful because one can use pictures to
encapsulate a large amount of information; complicated calculations involving compositions of morphisms can be visualized intuitively and written down
suffering only minor headaches.

Because of the simplicity of the diagrammatic calculus, we were able to calculate explicitly the chain maps which correspond to each generating
cobordism in the braid cobordism category, and check that these chain maps satisfy the same relations that braid cobordisms do. The general proofs
are straightforward and computationally explicit, performable by any reader with patience, time, and colored chalk. While we use some slightly more
sophisticated machinery to avoid certain incredibly lengthy computations, the machinery is completely unnecessary. This makes the results of
Rouquier and Khovanov that much more concrete, and implies the following new result.

\begin{theorem} \label{mainthm} There is a functor $F$ from the category of combinatorial braid cobordisms to the category of complexes in
$\mc{SC}_2$ up to homotopy, lifting Rouquier's construction (i.e. such that $F$ sends crossings to Rouquier complexes). \end{theorem}

Soergel bimodules are generally defined over certain fields $\Bbbk$ in the literature, because one is usually interested in Soergel bimodules as a
categorification of the Hecke algebra, and in relating indecomposable bimodules to the Kazhdan-Lusztig canonical basis. However, we invite the reader to
notice that the diagrammatic construction in \cite{EKh} can be made over any ring, and in particular over $\Z$. In fact, all our proofs of functoriality still
work over $\Z$. We discuss this in detail in section \ref{subsec-integers}. In the subsequent paper, we plan to use the work done here to define HOMFLY-PT and
$sl(n)$-link homology theories over $\Z$, a construction which is long overdue. We also plan to investigate the Rasmussen spectral sequence in this context.

At the given moment there does not exist a diagrammatic calculus for the higher Hochschild homology of Soergel bimodules. Some insights have already been
obtained, although a full understanding had yet to emerge. We plan to develop the complete picture, which should hopefully give an explicit and easily
computable description of functoriality in the link homology theories discussed above.

The organization of this paper is as follows. In Section \ref{sec-constructions} we go over all the previous constructions that are relevant to
this paper. This includes the Hecke algebra, Soergel's categorification $\mc{SC}$, the graphical presentation of $\mc{SC}$, the combinatorial
braid cobordism category, and Rouquier's complexes which link $\mc{SC}$ to braids. In Section \ref{subsec-conventions} we describe the
conventions we will use in the remainder of the paper to draw Rouquier complexes for movies. In Section \ref{sec-defn} we define the functor from
the combinatorial braid cobordism category to the homotopy category of $\mc{SC}$, and in Section \ref{sec-moviemoves} we check the movie move
relations to verify that our functor is well-defined. These checks are presented in numerical order, not in logical order, but a discussion of
the logical dependency of the proofs, and of the simplifications that are used, can be found in Section \ref{subsec-simplify}. Section
\ref{sec-additional} contains some useful statements for the interested reader, but is not strictly necessary. Some additional light is shed on
the generators and relations of $\mc{SC}$ in Section \ref{subsec-brute}, where it is demonstrated how the relations arise naturally from movie
moves. In Section \ref{subsec-integers} we briefly describe how one might construct the theory over $\Z$, so that future
papers may use this result to define link homology theories over arbitrary rings.
\section{Constructions}
\label{sec-constructions}

%
\subsection{The Hecke Algebra}
\label{subsec-hecke}
%
The Hecke algebra $\mc{H}$ of type $A_\infty$ has a presentation as an algebra over $\Ztt$ with
generators $b_i$, $i \in \Z$ and the \emph{Hecke relations}
\begin{eqnarray}
b_i^2 & = & (t + t^{-1}) b_i  \label{eqn-bisq}\\
b_ib_j & = & b_jb_i \ \mathrm{for}\  |i-j|\ge 2 \label{eqn-bibj}\\
b_ib_{i+1}b_i + b_{i+1} & = & b_{i+1}b_ib_{i+1} + b_i \label{eqn-bibpbi}.
\end{eqnarray}
For any subset $I \subset \Z$, we can consider the subalgebra $\mc{H}(I) \subset \mc{H}$
generated by $b_i$, $i \in I$, which happens to have the same presentation as above.  Usually only
finite $I$ are considered.

We write the monomial $b_{i_1}b_{i_2}\cdots b_{i_d}$ as $b_{\ii}$ where $\ii=i_1\ldots i_d$ is a
finite sequence of indices; by abuse of notation, we sometimes refer to this monomial simply as
$\ii$. If $\ii$ is as above, we say the monomial has \emph{length} $d$. We call a monomial \emph{non-repeating} if $i_k
\ne i_l$ for $k \ne l$. The empty set is a sequence of length 0, and $b_{\emptyset}=1$.

Let $\omega$ be the $t$-antilinear anti-involution which fixes $b_i$,
i.e. $\omega(t^ab_{\ii})=t^{-a}b_{\sigma(\ii)}$ where $\sigma$ reverses the order of a sequence. Let
$\epsilon \colon \mc{H} \to \Ztt$ be the $\Ztt$-linear map which is uniquely specified by
$\epsilon(xy)=\epsilon(yx)$ for all $x,y \in \mc{H}$ and $\epsilon(b_{\ii})=t^d$, whenever $\ii$ is a
non-repeating sequence of length $d$.  Let $(,) \colon \mc{H} \times \mc{H} \to \Ztt$ be the map
which sends $(x,y) \mapsto \epsilon(\omega(x)y)$. Via the inclusion maps, these structures all
descend to each $\mc{H}(I)$ as well.

We say $i,j \in \Z$ are \emph{adjacent} if $|i-j|=1$, and are \emph{distant} if $|i-j| \ge 2$.

For more details on the Hecke algebra in this context, see \cite{EKh}.

%
\subsection{The Soergel Categorification}
\label{subsec-soergel}
%

In \cite{Soe1}, Soergel introduced a monoidal category categorifying the
Hecke algebra for a finite Weyl group $W$ of type $A$. We will denote this category by $\mc{SC}(I)$, or by $\mc{SC}$ when $I$ is irrelevant. Letting $V$ be the geometric representation of $W$ over a field $\Bbbk$ of characteristic $\ne 2$, and $R$ its coordinate ring, the category $\mc{SC}$ is given as a full additive monoidal subcategory of graded
$R$-bimodules (whose objects are now commonly referred to as Soergel bimodules). This category is not
abelian, for it lacks images, kernels, and the like, but it is idempotent closed. In fact, $\mc{SC}$ is given
as the idempotent closure of another full additive monoidal subcategory $\mc{SC}_1$, whose objects are
called Bott-Samuelson modules. The category $\mc{SC}_1$ is generated monoidally over $R$ by objects $B_i$,
$i \in I$, which satisfy

\begin{equation} B_i \TenR B_i \cong B_i\{1\} \oplus B_i\{-1\} \label{dc-ii} \end{equation}
\begin{equation} B_i \TenR B_j \cong B_j \TenR B_i \label{dc-ij} \textrm{ for distant } i, j \end{equation}
\begin{equation} B_i \TenR B_j \TenR B_i \oplus B_j \cong B_j \TenR B_i \TenR B_j \oplus
  B_i \label{dc-ipi} \textrm{ for adjacent } i, j. \end{equation}

The Grothendieck group of $\mc{SC}(I)$ is isomorphic to $\mc{H}(I)$, with the class of $B_i$ being
sent to $b_i$, and the class of $R\{1\}$ being sent to $t$.

One useful feature of this categorification is that it is easy to calculate the dimension of Hom spaces in each degree. Let $\HOM(M,N) \define
\bigoplus_{m \in \Z} \Hom(M,N\{m\})$ be the graded vector space (actually an $R$-bimodule) generated by homogeneous morphisms of all degrees. Let
$B_{\ii} \define B_{i_1} \TenR \cdots \TenR B_{i_d}$. Then $\HOM(B_{\ii},B_{\jj})$ is a free left $R$-module, and its graded rank over $R$ is
given by $(b_{\ii},b_{\jj})$.

For two subsets $I \subset I^\prime \subset \Z$, the categories $\mc{SC}(I)$ and $\mc{SC}(I^\prime)$
are embedded in bimodule categories over different rings $R(I)$ and $R(I^\prime)$, but there is nonetheless a faithful
inclusion of categories $\mc{SC}(I) \to \mc{SC}(I^\prime)$.  This functor is not full: the size of
$R$ itself will grow, and $\HOM(B_{\emptyset},B_{\emptyset})=R$. However, the graded rank over $R$
does not change, since the value of $\epsilon$ and hence $(,)$ does not change over various
inclusions.  Effectively, the only difference in Hom spaces under this inclusion functor is base
change on the left, from $R(I)$ to $R(I^\prime)$.

As a result of this, most calculations involving morphisms between Soergel bimodules will not depend on which $I$ we
work over.  When $I$ is infinite, the ring $R$ is no longer Noetherian, and we do not wish to deal
with such cases.  However, the categories $\mc{SC}(I)$ over arbitrary finite $I$ will all work
essentially the same way.  A slightly more rigorous graphical statement of this property is
forthcoming.  In particular, the calculations we do for the Braid group on $m$ strands will also
work for the braid group on $m+1$ strands, and so forth.

%
\subsection{Soergel Diagrammatics}
\label{subsec-diagrammatics}
%

In \cite{EKh}, the category $\mc{SC}_1$ was given a diagrammatic presentation by generators and
relations, allowing morphisms to be viewed as isotopy classes of certain graphs. We
review this presentation here, referring the reader to \cite{EKh} for more details.  We will first
deal with the case where $W=S_{n+1}$, or where $I=\{1,2,\ldots,n\}$, and then discuss what the inclusions
of categories from the previous section imply for the general setting.

\begin{remark} Technically, \cite{EKh} gave the presentation for a slightly different category, which we
temporarily call $\mc{SC}_1^\prime$. The category presented here is a quotient of $\mc{SC}_1^\prime$ by the
central morphism corresponding to $e_1$, the first symmetric polynomial. This is discussed briefly in
Section 4.5 of \cite{EKh}. Moreover, $\mc{SC}_1^\prime$ is also a faithful extension of $\mc{SC}_1$, so
that the main results of this paper apply to the extension as well. We use $\mc{SC}_1$ instead because it
is the ``minimal" category required for our results (no extensions are necessary), and because it
streamlines the presentation. We leave it as an exercise to see that the definition of $\mc{SC}_1$ below agrees with
the $e_1$ quotient of the category defined in \cite{EKh}.

The first subtlety to be addressed is that $\mc{SC}_1$ is only equivalent to the $e_1$ quotient of $\mc{SC}_1^\prime$ when one is working over a
base ring $\Bbbk$ where $n+1$ is invertible. Otherwise, the quotient of $\mc{SC}_1^\prime$ is still a non-trivial faithful extension.

For a discussion of the advantages to using $\mc{SC}_1^\prime$, see Section \ref{subsec-integers}. \end{remark}

An object in $\mc{SC}_1$ is given by a sequence of indices $\ii$, which is visualized as $d$ points
on the real line $\R$, labelled or ``colored'' by the indices in order from left to right. Sometimes
these objects are also called $B_{\ii}$.  Morphisms are given by pictures embedded in the strip $\R
\times [0,1]$ (modulo certain relations), constructed by gluing the following generators
horizontally and vertically:

\igc{1}{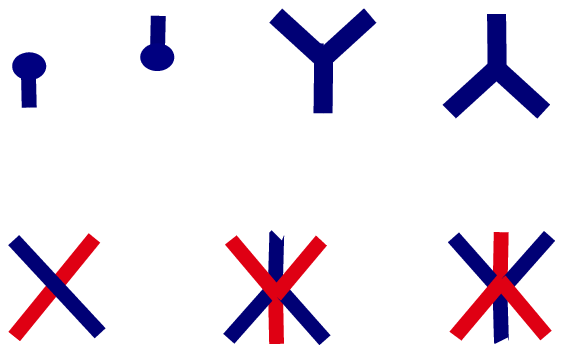}

For instance, if "blue" corresponds to the index $i$ and "red" to $j$, then the lower right generator is a morphism from $jij$ to $iji$. The generating
pictures above may exist in various colors, although there are some restrictions based on of adjacency conditions.

We can view a morphism as an embedding of a planar graph, satisfying the following properties:
\begin{enumerate}
\item Edges of the graph are colored by indices from $1$ to $n$.
\item Edges may run into the boundary $\R \times \{0,1\}$, yielding two sequences of colored points
  on $\R$, the top boundary $\ii$ and the bottom boundary $\jj$.  In this case, the graph is viewed
  as a morphism from $\jj$ to $\ii$.
\item Only four types of vertices exist in this graph: univalent vertices or ``dots'', trivalent
  vertices with all three adjoining edges of the same color, 4-valent vertices whose adjoining edges
  alternate in colors between $i$ and $j$ distant, and 6-valent vertices whose adjoining
  edges alternate between $i$ and $j$ adjacent.
\end{enumerate}

The degree of a graph is +1 for each dot and -1 for each trivalent vertex. $4$-valent and $6$-valent vertices are of degree $0$. The term \emph{graph} henceforth refers to such a graph embedding.

By convention, we color the edges with different colors, but do not specify which colors match up with which $i \in I$. This is legitimate, as
only the various adjacency relations between colors are relevant for any relations or calculations. We will specify adjacency for all pictures,
although one can generally deduce it from the fact that 6-valent vertices only join adjacent colors, and 4-valent vertices join only distant colors.

As usual in a diagrammatic category, composition of morphisms is given by vertical concatenation,
and the monoidal structure is given by horizontal concatenation.

In writing the relations, it will be useful to introduce a pictures for the ``cup'' and ``cap'':

\begin{equation} \ig{1}{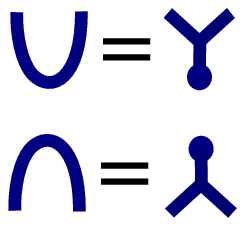} \end{equation}

We then allow $\Bbbk$-linear sums of graphs, and apply the relations below to obtain our category $\mc{SC}_1$. Some of these relations are
redundant. For a more detailed discussion of the remarks in the remainder of this section, see \cite{EKh}.

\begin{equation} \label{twistline} \ig{.9}{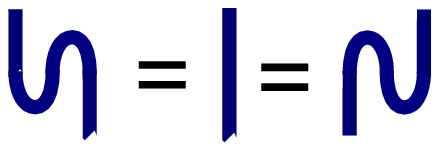}  \end{equation}
\begin{equation} \label{twistdot} \ig{.9}{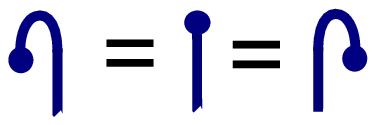}\end{equation}
\begin{equation} \label{twist3} \ig{.9}{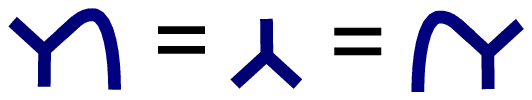}\end{equation}
\begin{equation} \label{twist4} \ig{.9}{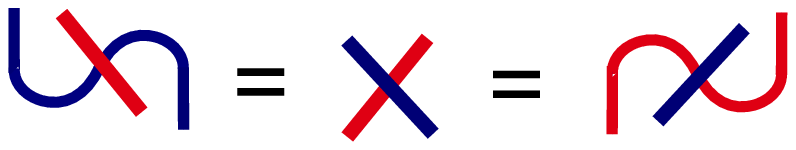} \end{equation}
\begin{equation} \label{twist6} \ig{.9}{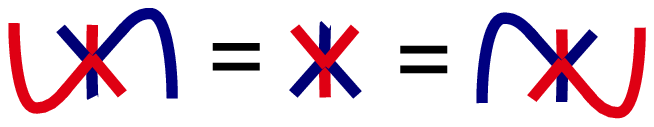} \end{equation}

\begin{remark} The relations (\ref{twistline}) through (\ref{twist6}) together imply that the morphism
specified by a particular graph embedding is independent of the isotopy class of the embedding. We could
have described the category more simply by defining a morphism to be an isotopy class of a certain kind of
planar graph. However, it is useful to understand that these ``isotopy relations'' exist, because they will
appear naturally in the study of movie moves (see Section \ref{subsec-brute}).

Other relations are written in a format which already assumes that isotopy invariance is given. Some of these relations contain horizontal lines, which
cannot be constructed using the generating pictures given; nonetheless, such a graph is isotopic to a number of different pictures which are indeed
constructible, and it is irrelevant which version you choose, so the relation is unambiguous. \end{remark}

\begin{equation} \label{assoc1} \ig{.9}{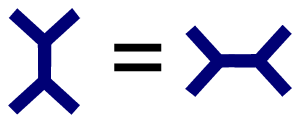} \end{equation}

\begin{remark} Relation (\ref{assoc1}) effectively states that a certain morphism is invariant under 90 degree rotation. To simplify drawings
later on, we often draw this morphism as follows:
\igc{.4}{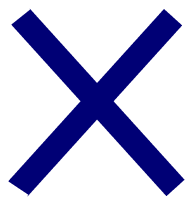}
 Note that morphisms will still be isotopy invariant with this convention. \end{remark}

Here are the remainder of the one color relations.

\begin{equation} \label{unit} \ig{.9}{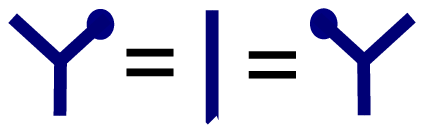} \end{equation}
\begin{equation} \label{needle} \ig{1.4}{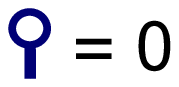} \end{equation}
\begin{equation} \label{dotslidesame} \ig{1.5}{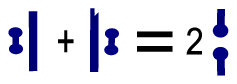} \end{equation}

In the following relations, the two colors are distant.

\begin{equation} \label{R2} \ig{.6}{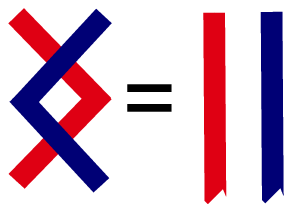} \end{equation}
\begin{equation} \label{distslidedot} \ig{1}{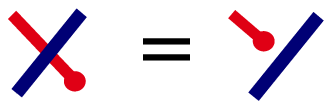} \end{equation}
\begin{equation} \label{distslide3} \ig{.7}{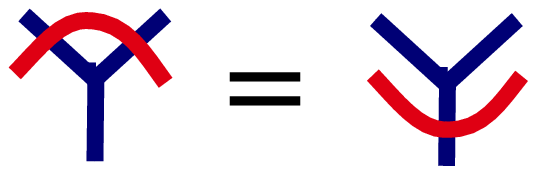} \end{equation}
\begin{equation} \label{dotslidefar} \ig{.6}{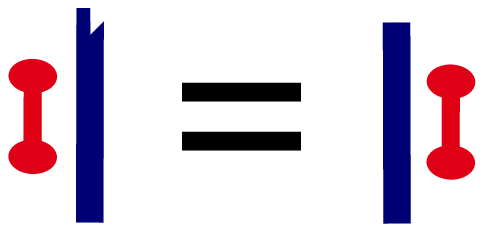} \end{equation}

In this relation, two colors are adjacent, and both distant to the third color.

\begin{equation} \label{distslide6} \ig{.7}{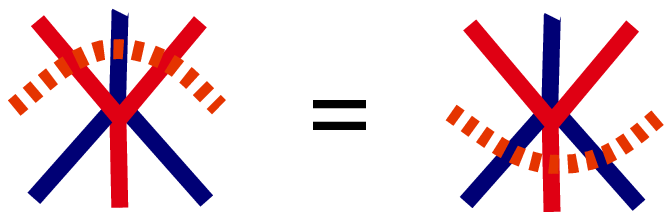} \end{equation}

In this relation, all three colors are mutually distant.

\begin{equation} \label{distslide4} \ig{1}{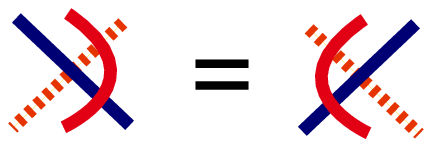} \end{equation}

\begin{remark} Relations (\ref{R2}) thru (\ref{distslide4}) indicate that any part of the graph colored $i$ and any part of the graph colored $j$
``do not interact'' for $i$ and $j$ distant. That is, one may visualize sliding the $j$-colored part past the $i$-colored part, and it will not
change the morphism. We call this the \emph{distant sliding property}. \end{remark}

In the following relations, the two colors are adjacent.

\begin{equation} \label{dot6} \ig{.7}{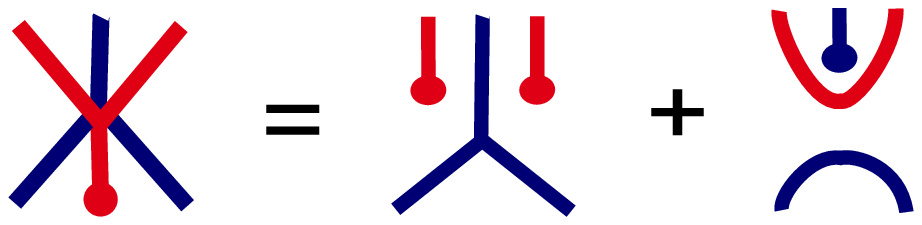} \end{equation}
\begin{equation} \label{ipidecomp} \ig{.7}{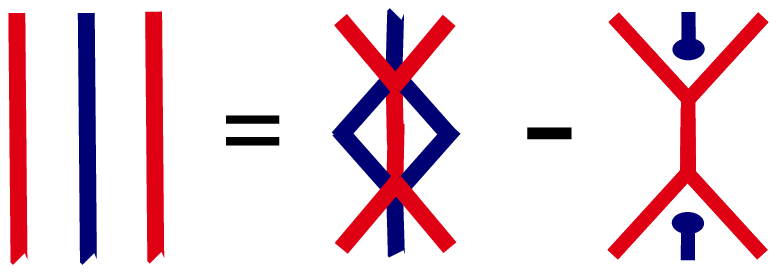} \end{equation}
\begin{equation} \label{assoc2} \ig{.7}{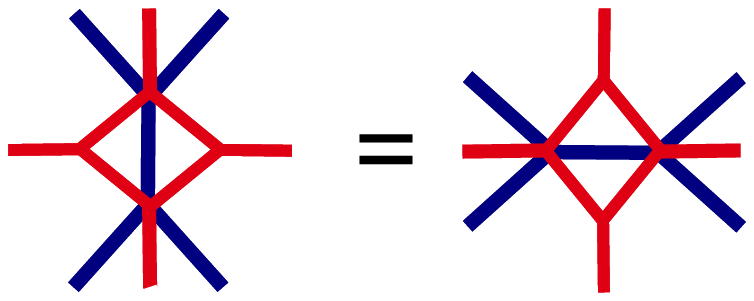} \end{equation}
\begin{equation} \label{dotslidenear} \ig{.7}{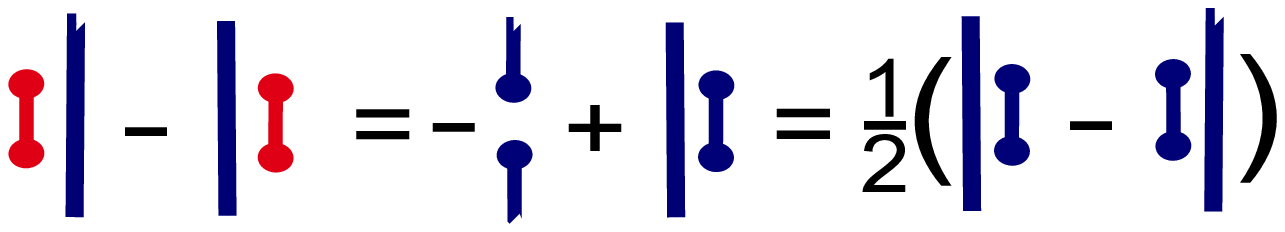} \end{equation}

The last equality in (\ref{dotslidenear}) is implied by (\ref{dotslidesame}), so it is not necessary to include as a relation. In this final
relation, the colors have the same adjacency as $\{1,2,3\}$.

\begin{equation} \label{assoc3} \ig{.7}{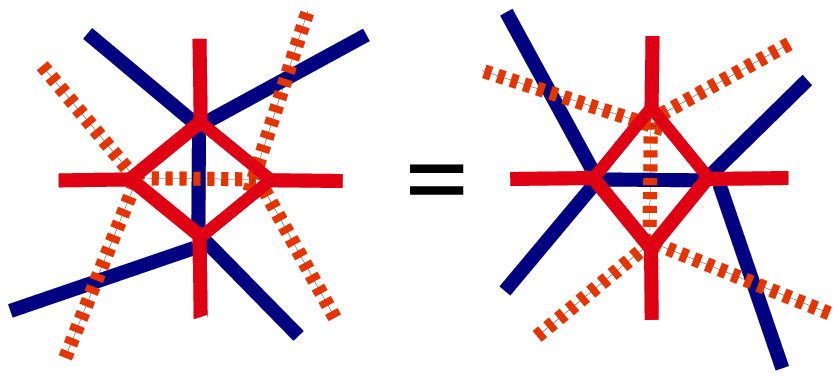} \end{equation}

\begin{remark} Because of isotopy invariance, the object $B_i$ in $\mc{SC}_1$ is self-biadjoint. In
  particular, instead of viewing the graph in $\R \times [0,1]$ as a morphism from $\ii$ to $\jj$,
  we could twist it around and view it in the lower half plane (with no bottom boundary) as a morphism from
  $\emptyset$ to $\ii\sigma(\jj)$. Thus, we need only investigate morphisms from
  $\emptyset$ to $\ii$, to determine all Hom spaces.
\end{remark}

\begin{remark} There is a functor from this category into the category of $R$-bimodules, sending a
  line colored $i$ to $B_i$ and each generator to an appropriate bimodule map. The functor gives an
  equivalence of categories between this graphically defined category and the subcategory
  $\mc{SC}_1$ of $R$-bimodules mentioned in the previous section, so the use of the same name is legitimate.
\end{remark}

We refer to any connected component of a graph which is a dot connected directly to the boundary as
a \emph{boundary dot}, and to any component equal to two dots connected by an edge as a \emph{double
  dot}.

\begin{remark} Relations (\ref{dotslidesame}), (\ref{dotslidefar}), and (\ref{dotslidenear}) are collectively called \emph{dot slides}. They
indicate how one might attempt to move a double dot from one region of the graph to another. \end{remark}

The following theorem and corollary are the most important results from \cite{EKh}, and the crucial
fact which allows all other proofs to work.

\begin{theorem} \label{colorreduction} Consider a morphism $\phi \colon \ii \to \emptyset$, and suppose that the index $i$
  appears in $\ii$ zero times (respectively, once). Then $\phi$ can be rewritten as a linear
  combination of graphs, for which each graph has the following property: the only edges of the
  graph colored $i$ are included in double dots (respectively, as well as a single boundary dot
  connecting to $\ii$), and moreover, all these double dots are in the leftmost region of the
  graph. This result may be obtained simultaneously for multiple indices $i$. We could also have chosen the rightmost region for the slide.
\end{theorem}

\begin{cor} The space $\HOM_{\mc{SC}_1}(\emptyset,\emptyset)$ is the free commutative polynomial
  ring generated by $f_i$, the double dot colored $i$, for various $i \in I$. This is a graded ring, with the degree of $f_i$ is 2. 
\end{cor}

\begin{remark} Another corollary of the more general results in \cite{EKh} is that, when a color only appears twice in the boundary one can
(under certain conditions on other colors present) reduce the graph to a form where that color only appears in a line connecting the two boundary
appearances (and double dots as usual). In particular, if no color appears more than twice on the boundary, then under certain conditions one can
reduce all graphs to a form that has no trivalent vertices, and hence all morphisms have nonnegative degree. We will use this
fact to help check movie moves 8 and 9, in whose contexts the appropriate conditions do hold. \end{remark}

The proof of this theorem involves using the relations to reduce a single color at a time within a graph (while doing arbitrary things to the other
colors). Once a color is reduced to the above form, the remainder of the graph no longer interacts with that color. Then we repeat
the argument with another color on the rest of the graph, and so on and so forth.

\begin{remark} There is a natural identification of the polynomial ring of double dots and the coordinate ring $R$ of the geometric
representation. Because of this, a combination of double dots is occasionally referred to as a polynomial. Placing double dots in the lefthand or
righthand region of a diagram will correspond to the left and right action of $R$ on Hom spaces. \end{remark}

\begin{remark} Now we are in a position to see how the inclusion $\mc{SC}_1(I) \subset \mc{SC}_1(I^\prime)$ behaves. Let $\ii$ and $\jj$ be objects in
$\mc{SC}_1(I)$, and $k$ an index in $I^\prime \setminus I$. Applying Theorem \ref{colorreduction} to the color $k$, we can assume that in
$\mc{SC}_1(I^\prime)$ all morphisms from $\ii$ to $\jj$ will be (linear combinations of) graphs where $k$ only appears in double dots on the left. Doing this
to each color in $I^\prime \setminus I$, we will have a collection of double dots next to a morphism which only uses colors in $I$. Therefore the map
$\HOM_{\mc{SC}(I)}(\ii,\jj) \otimes \Bbbk[f_k,\ k \in I^\prime \setminus I ] \to \HOM_{\mc{SC}(I^\prime)}(\ii,\jj)$ is surjective. In fact, it is an
isomorphism. We say that the inclusion functor is \emph{fully faithful up to base change}. Of course, this result does not make it any easier to take a graph, which may have an arbitrarily complicated $k$-colored part, and reduce it to the simple form where $k$ only appears in double dots on the left.

If we wished to define $\mc{SC}_1(I)$ for some $I \subset \{1,\ldots,n\}$, the correct definition would be to consider graphs which
are only colored by indices in $I$. With this definition, inclusion functors are still fully faithful up to base change. \end{remark}

Now we see where the isomorphisms (\ref{dc-ii}) through (\ref{dc-ipi}) come from. To begin, we have the following implication of
(\ref{dotslidesame}):

\begin{equation} \label{iidecomp} \ig{1.7}{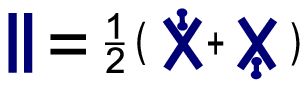} \end{equation}

We let $\mc{SC}_2$ be the category formally containing all direct sums and grading shifts of objects in $\mc{SC}_1$, but whose morphisms are
forced to be degree 0. Then (\ref{iidecomp}) expresses the direct sum decomposition \begin{equation} B_i \TenR B_i = B_i\{1\} \oplus B_i\{-1\}
\nonumber \end{equation} since it decomposes the identity $\id_{ii}$ as a sum of two orthogonal idempotents, each of which is the composition of a
projection and an inclusion map of the appropriate degree. If one does not wish to use non-integral coefficients, and an adjacent color is
present, then the following implication of (\ref{dotslidenear}) can be used instead; this is again a decomposition of $\id_{ii}$ into orthogonal
idempotents.

\begin{equation} \label{iidecomp2} \ig{.8}{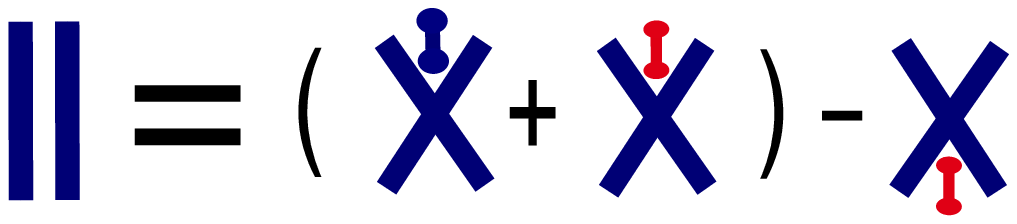} \end{equation}

Relation (\ref{R2}) expresses the
isomorphism \begin{equation} B_i \TenR B_j = B_j \TenR B_i \nonumber \end{equation} for $i$ and $j$ distant.

The category $\mc{SC}$ is the Karoubi envelope, or idempotent completion, of the category
$\mc{SC}_2$. Recall that the Karoubi envelope of a category $\mc{C}$ has as objects pairs $(B,e)$
where $B$ is an object in $\mc{C}$ and $e$ an idempotent endomorphism of $B$.  This object acts as
though it were the ``image'' of this projection $e$, and in an additive category behaves like a
direct summand. For more information on Karoubi envelopes, see Wikipedia.

The two color variants of relation (\ref{ipidecomp}) together express the direct sum
decompositions 
\begin{eqnarray} B_i \TenR B_{i+1} \TenR B_i = C_i \oplus B_i \\ B_{i+1} \TenR B_i \TenR
B_{i+1} = C_i \oplus B_{i+1}. \end{eqnarray}
 Again, the identity $\id_{i(i+1)i}$ is decomposed into
orthogonal idempotents, where the first idempotent corresponds to a new object $C_i$ in the idempotent
completion, appearing as a summand in both $i(i+1)i$ and $(i+1)i(i+1)$. Technically, we get two new
objects, corresponding to the idempotent in $B_{i(i+1)i}$ and the idempotent in $B_{(i+1)i(i+1)}$, but
these two objects are isomorphic, so by abuse of notation we call them both $C_i$.

We will primarily work within the category $\mc{SC}_2$. However, since this includes fully
faithfully into $\mc{SC}$, all calculations work there as well.

%
\subsection{Braids and Movies}
\label{subsec-braids}
%

In this paper we always use the combinatorial braid cobordism category as a replacement for the topological braid cobordism category, since they are
equivalent but the former is more convenient for our purposes. See Carter and Saito \cite{CS} for more details.

The category of $(n+1)$-stranded braid cobordisms can be defined as follows. The objects are arbitrary sequences of braid group generators $O_i$,
$1\leq i \leq n$, and their inverses $U_i = O_i^{-1}$. These sequences can be drawn using braid diagrams on the plane, where $O_i$ is an
overcrossing (the $i+1\st$ strand crosses over the $i\th$ strand) and $U_i$ is an undercrossing. Objects have a monoidal structure given by
concatenation of sequences. A \emph{movie} is a finite sequence of transformations of two types:\\

I. Reidemeister type moves, such as

$$\tau_1  O_i U_i \tau_2 \leftrightarrow \tau_1 \tau_2,$$

$$\tau_1 O_i O_j \tau_2       \leftrightarrow \tau_1  O_j O_i \tau_2 \textrm{ for distant } i, j $$

$$\tau_1 O_i O_{i+1} O_i \tau_2       \leftrightarrow \tau_1  O_{i+1} O_i O_{i+1} \tau_2. $$

where $\tau_1$ and $\tau_2$ are arbitrary braid words.\\ 

II. Addition or removal of a single $O_i$ or $U_i$ from a braid word

$$\tau_1 \tau_2 \leftrightarrow \tau_1 O_i^{\pm 1} \tau_2.$$

These transformations are known as \emph{movie generators}. Morphisms in this category will consist of movies modulo \emph{locality moves}, which
ensure that the category is a monoidal category, and certain relations known as \emph{movie moves} (it is common also to refer to locality moves as
movie moves). The movie moves can be found in figures \ref{MMa} and \ref{MMb}. Movie moves $1-10$ are composed of type I transformations and $11-14$
each contains a unique type II move. We denote the location of the addition or removal of a crossing in these last $4$ movies by little black
triangles. There are many variants of each of these movies: one can change the relative height of strands, can reflect the movie horizontally or
vertically, or can run the movie in reverse. We refer the reader to Carter and Saito \cite{CS}, section $3$.

\begin{figure} [!htbp]
\centerline{
\includegraphics[scale=.8]{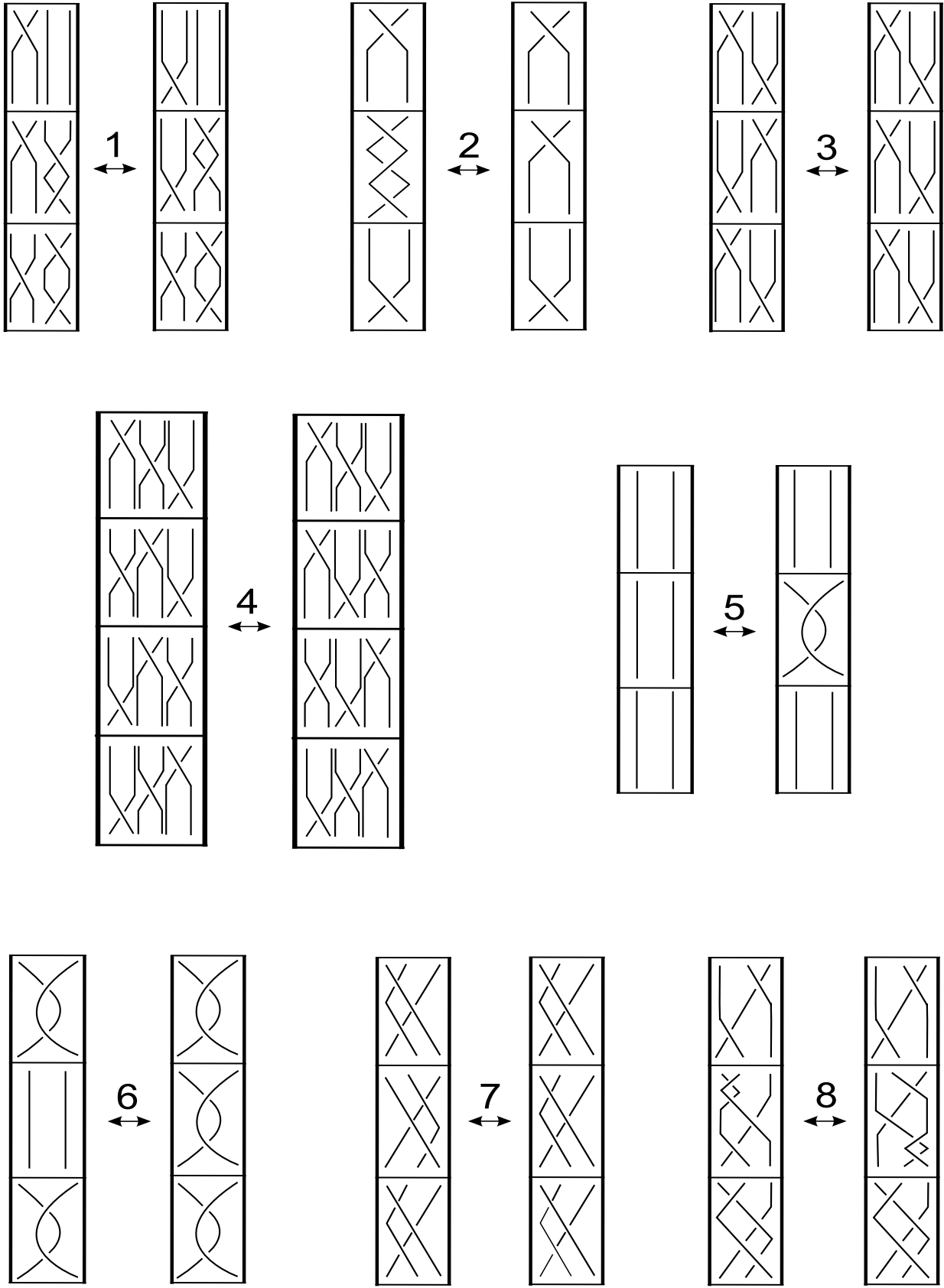}}
\caption{Braid movie moves $1-8$} \label{MMa}
\end{figure}\

\begin{figure} [!htbp]
\centerline{
\includegraphics[scale=.8]{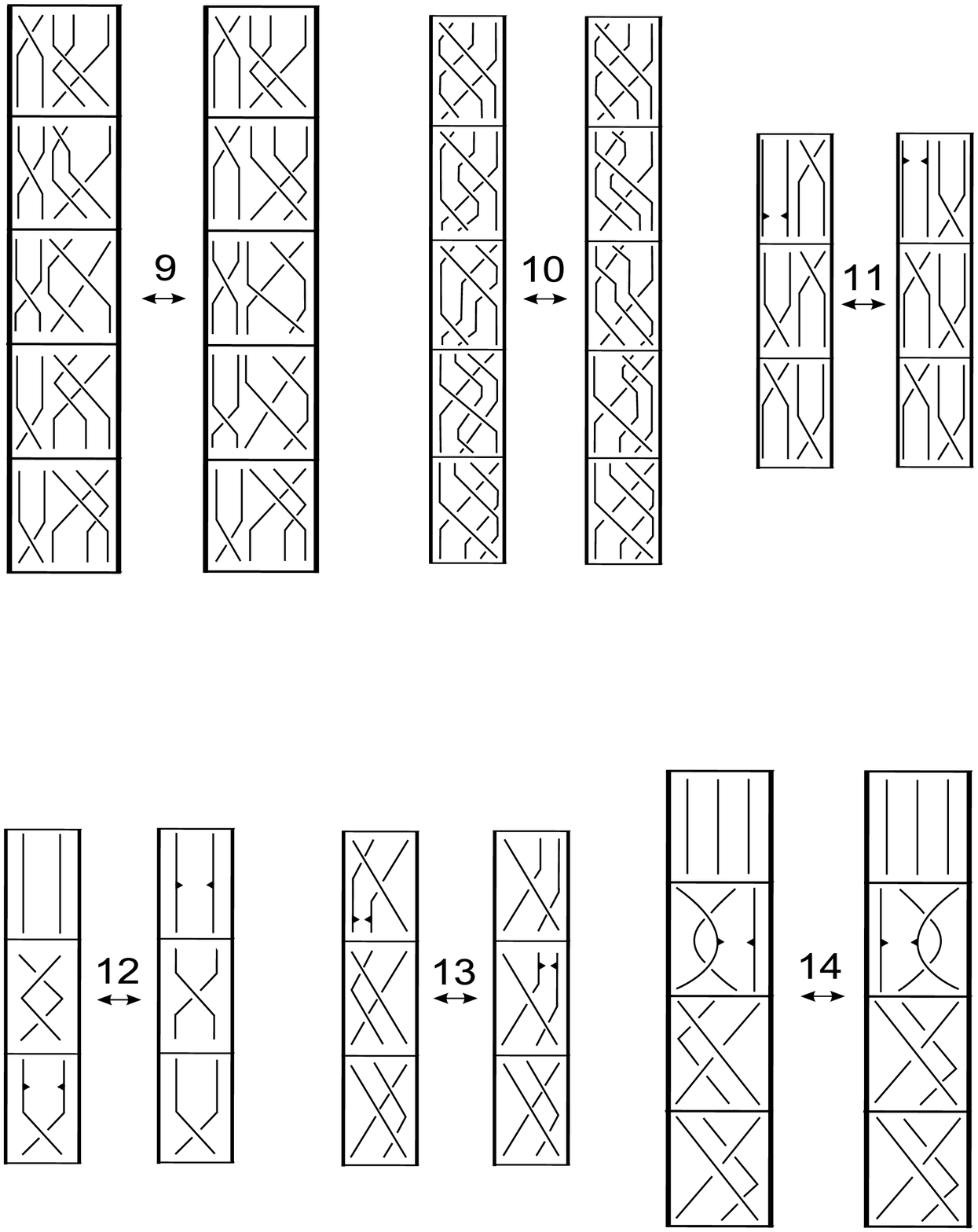}}
\caption{Braid movie moves $9-14$} \label{MMb}
\end{figure}\

Recall that the combinatorial cobordism category is monoidal. Locality moves merely state that if two transformations are performed on a diagram in locations
that do not interact (they do not share any of the same crossings) then one may change the order in which the transformations are performed. Any
potential functor from the combinatorial cobordism category to a monoidal category $\mc{C}$ which preserves the monoidal structure will
automatically satisfy the locality moves. Because of this, we need not mention the locality moves again.

\begin{defn} Given a braid diagram $P$ (or an object in the cobordism category), the diagram $\overline{P}$ is given by reversing the sequence
defining $P$, and replacing all overcrossings with undercrossings and vice versa. \end{defn}

Note that $\overline{P}$ is the inverse of $P$ in the group generated freely by crossings, and hence in the braid group as well.

Again, we refer the reader to \cite{CS} for more details on the combinatorial braid cobordism category.

\pagebreak

%
\subsection{Rouquier Complexes}
\label{subsec-rouquier}
%

Rouquier defined a braid group action on the homotopy category of complexes in $\mc{SC}_2$ (see
\cite{Rou1}). To the $i\th$ overcrossing, he associated a complex $B_i\{1\} \longrightarrow
B_{\emptyset}$, and to the undercrossing, $B_{\emptyset} \longrightarrow B_i\{-1\}$. In each case,
$B_{\emptyset}$ is in homological degree 0.  Drawn graphically, these complexes look like:

\begin{figure} [!htbp]
\centerline{
\includegraphics[scale=.6]{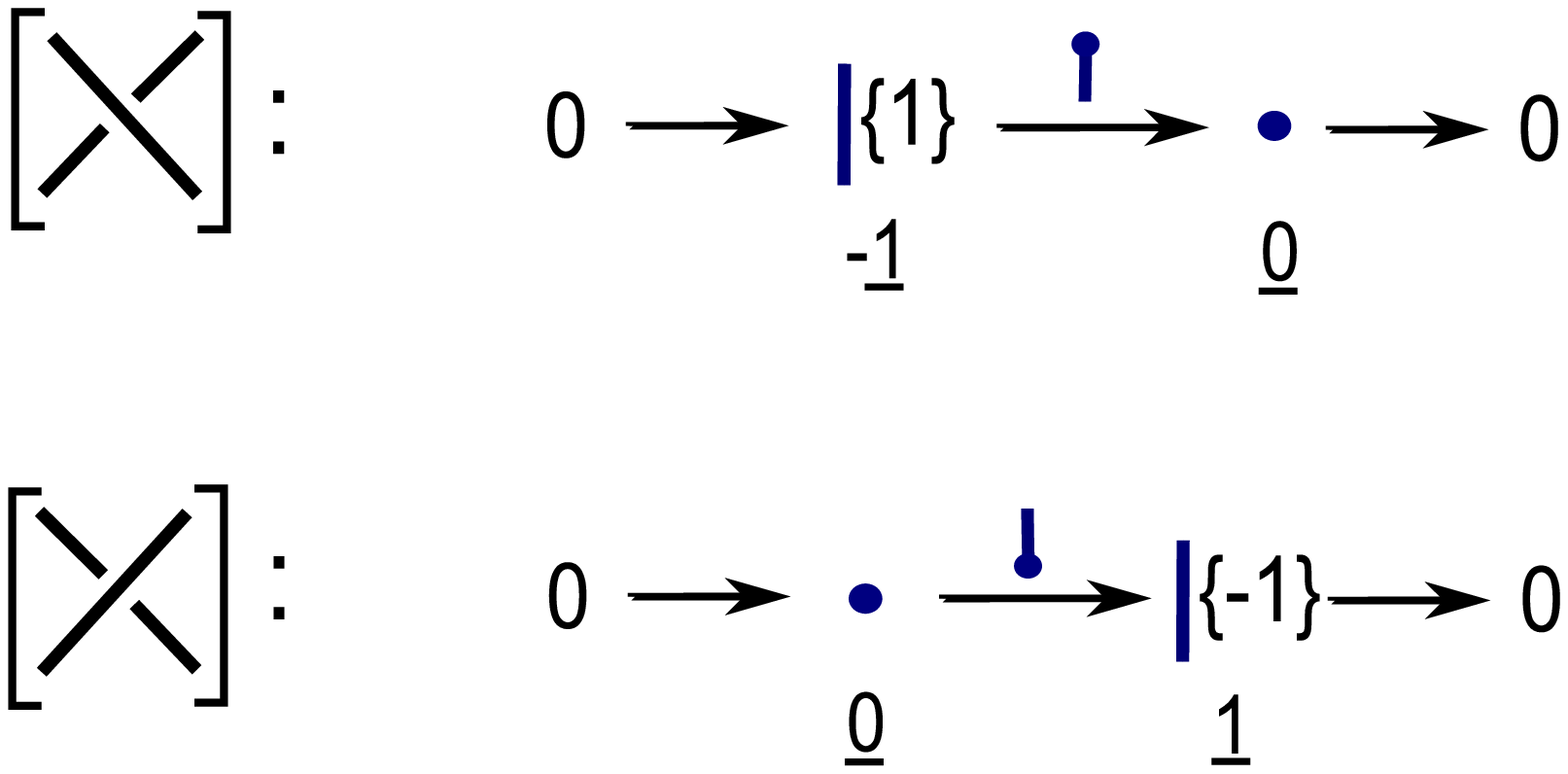}}
\caption{Rouquier complex for right and left crossings} \label{crossingsdef}
\end{figure}\

We are using a (blue) dot here as a place holder for empty space.

To a braid one associates the tensor product of the complexes for each crossing.  He showed in \cite{Rou1}
that the braid relations hold amongst these complexes.

In \cite{K1}, Khovanov showed that taking Hochschild cohomology of these complexes yields
an invariant of the link which closes off the braid in question, and that this link homology theory
is in fact identical to one already constructed by Khovanov and Rozansky in \cite{KR}.  It was
shown in \cite{KT} that Rouquier's association of complexes to a braid is actually \emph{projectively
  functorial}.  In other words, to each movie between braids, there is a map of complexes, and these
maps satisfy the movie move relations (modulo homotopy) up to a potential sign. This was not done by
explicitly constructing chain maps, but instead used the formal consequences of the
previously-defined link homology theory. It was known that in many cases the composed map would be
an isomorphism, and that this categorification could be done over $\Z$ (see \cite{KT}), where the only
isomorphisms are $\pm 1$, hence the proof of projective functoriality.

The discussion of the previous sections shows that it is irrelevant which braid group we work in,
because adding extra strands just corresponds to an inclusion functor which is ``fully faithful
after base change''.  In particular, when computing the space of chain maps modulo homotopy between
two complexes, we need not worry about the number of strands available, except to keep track of our
base ring.  Hence calculations are effectively local.

%
\subsection{Conventions}
\label{subsec-conventions}
%

These are the conventions we use to draw Rouquier complexes henceforth.

We use a colored circle to indicate the empty graph, but maintain the color for reasons of sanity.
It is immediately clear that in the complex associated to a tensor product of $d$ Rouquier complexes, each
summand will be a sequence of $k$ lines where $0 \leq k \leq d$ (interspersed with colored
circles, but these represent the empty graph so could be ignored). Each differential from one
summand to another will be a ``dot'' map, with an appropriate sign.

\begin{enumerate}

 \item The dot would be a map of degree 1 if $B_i$ had not been shifted accordingly. In $\mc{SC}_2$, all maps must be homogeneous, so we could have deduced
the degree shift in $B_i$ from the degree of the differential. Because of this, it is not useful to keep track of various degree shifts of objects in a
complex. We will draw all the objects without degree shifts, and all differentials will therefore be maps of graded degree 1 (as well as homological degree
1). It follows from this that homotopies will have degree -1, in order to be degree 0 when the shifts are put back in. One could put in the degree shifts
later, noting that $B_{\emptyset}$ always occurs as a summand in a tensor product exactly once, with degree shift 0.

\item Similarly, one need not keep track of the homological dimension. $B_{\emptyset}$ will always
  occur in homological dimension $0$.

\item We will use blue for the index associated to the leftmost crossing in the braid, then red and
  dotted orange for other crossings, from left to right. The adjacency of these various colors is
  determined from the braid.

\item We read tensor products in a braid diagram from bottom to top.  That is, in the following
  diagram, we take the complex for the blue crossing, and tensor by the complex for the red
  crossing. Then we translate this into pictures by saying that tensors go from left to right. In
  other words, in the complex associated to this braid, blue always appears to the left of red.

\vspace{3mm}

\begin{figure} [!htbp]
\centerline{
\includegraphics[scale=1.2]{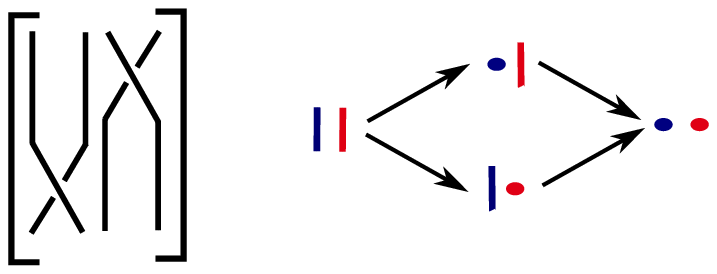}}
\label{example1}
\end{figure}\

\vspace{-5mm}

\item One can deduce the sign of a differential between two summands using the Liebnitz rule,
  $d(ab)=d(a)b + (-1)^{|a|}ad(b)$.  In particular, since a line always occurs in the basic complex
  in homological dimension $\pm 1$, the sign on a particular differential is exactly given by the
  parity of lines appearing to the left of the map.  For example,

\vspace{3mm}

\begin{figure} [!htbp]
\centerline{
\includegraphics[scale=.9]{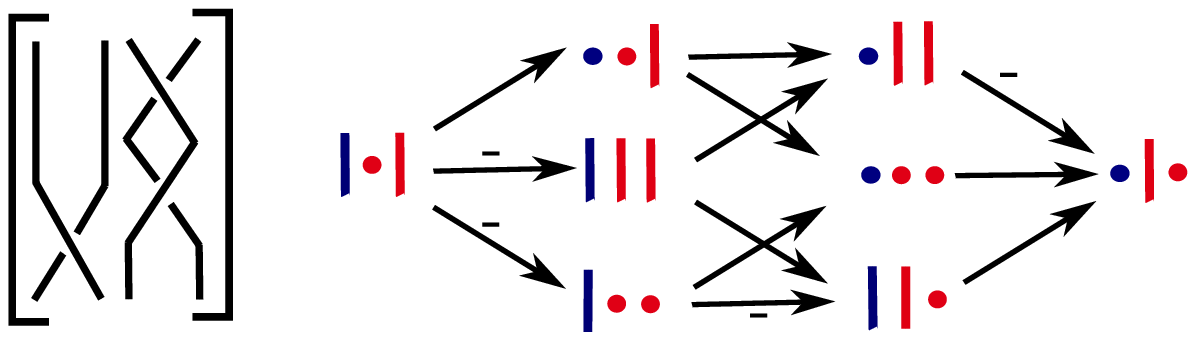}}
\label{example2}
\end{figure}\

\vspace{-5mm}

\item When putting an order on the summands in the tensored complex, we use the following
  standardized order.  Draw the picture for the object of smallest homological degree, which we draw
  with lines and circles.  In the next homological degree, the first summand has the first color
  switched (from line to circle, or circle to line), the second has the second color switched, and
  so forth.  In the next homological degree, two colors will be switched, and we use the
  lexicographic order: 1st and 2nd, then 1st and 3rd, then 1st and 4th... then 2nd and 3rd,
  etc. This pattern continues.

\vspace{3mm}

\begin{figure} [!htbp]
\centerline{
\includegraphics[scale=1.1]{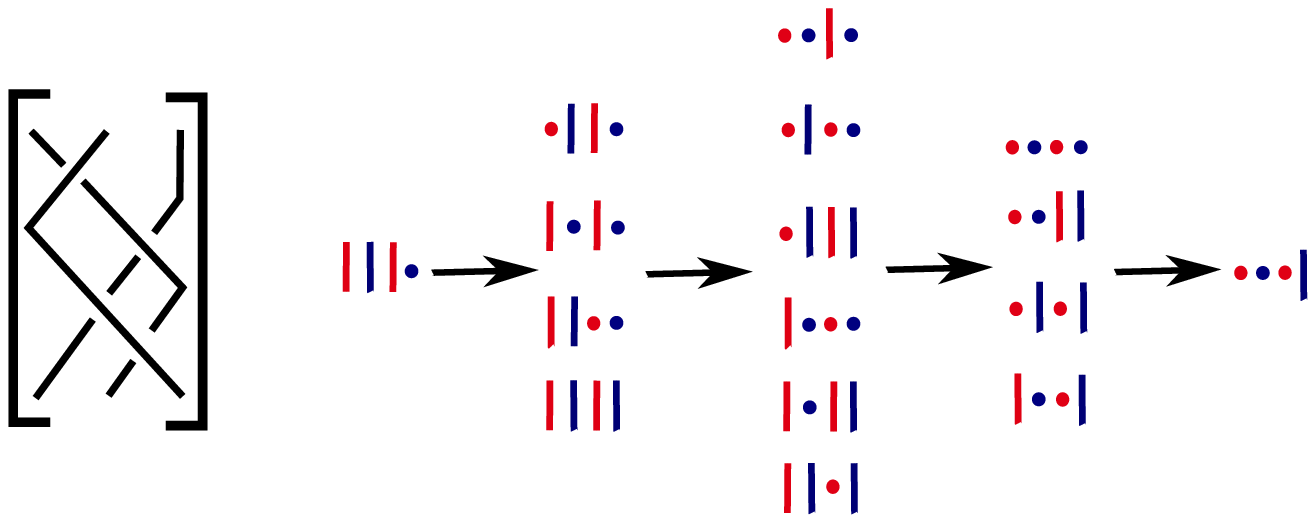}}
\label{example3}
\end{figure}\

\end{enumerate}

\pagebreak

\section{Definition of the Functor}
\label{sec-defn}

We extend Rouquier's complexes to a functor $F$ from the combinatorial braid cobordism category to the category of chain complexes in $\mc{SC}_2$ modulo homotopy.  Rouquier already defined how the functor acts on objects, so it only remains to define chain maps for each of the movie generators, and check the movie move relations.

There are four basic types of movie generators: birth/death of a crossing, slide, Reidemeister $2$ and Reidemeister $3$.

\begin{itemize}
\item \textbf{Birth and Death generators}
\begin{figure} [!htbp]
\centerline{
\includegraphics[scale=.9]{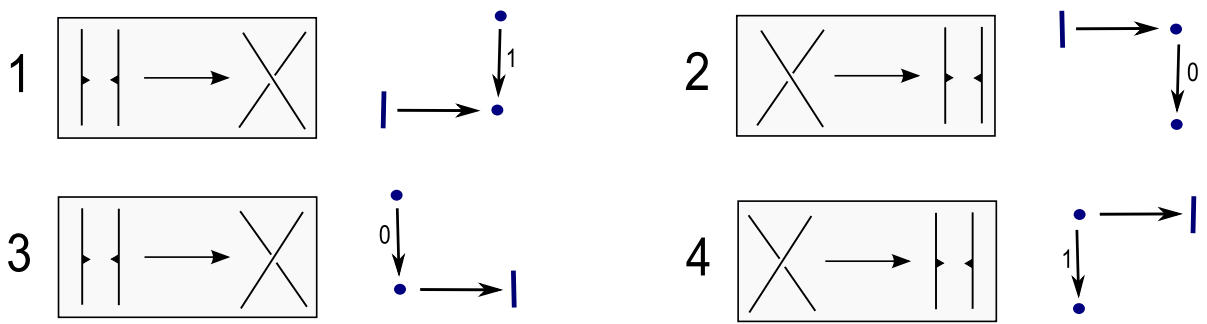}}
\caption{Birth and Death of a crossing generators} \label{MMgenBD}
\end{figure}\

\item \textbf{Reidemeister $2$ generators}
\begin{figure} [!htbp]
\centerline{
\includegraphics[scale=.9]{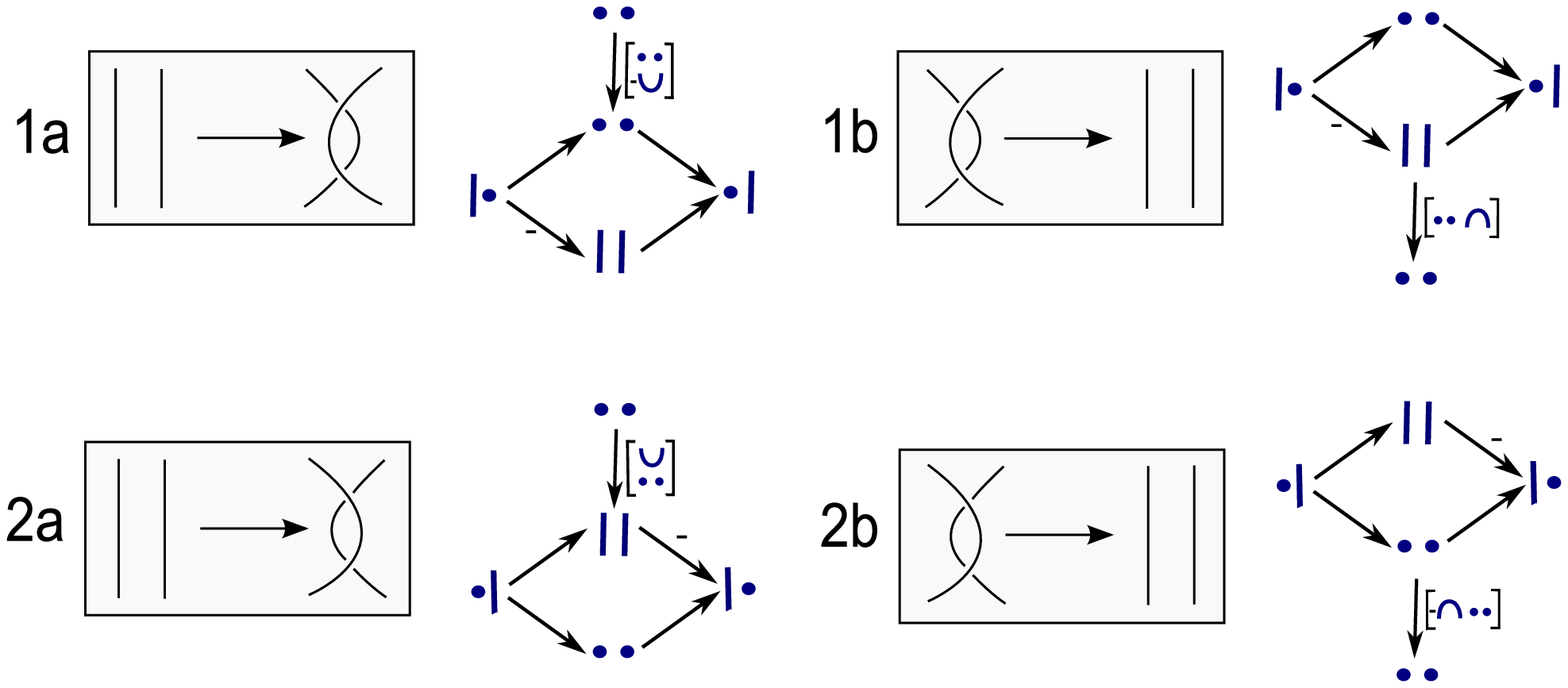}}
\caption{Reidemeister $2$ type movie move generators} \label{MMgenR2}
\end{figure}\

\pagebreak 

\item \textbf{Slide generators}
\begin{figure} [!htbp]
\centerline{
\includegraphics[scale=.9]{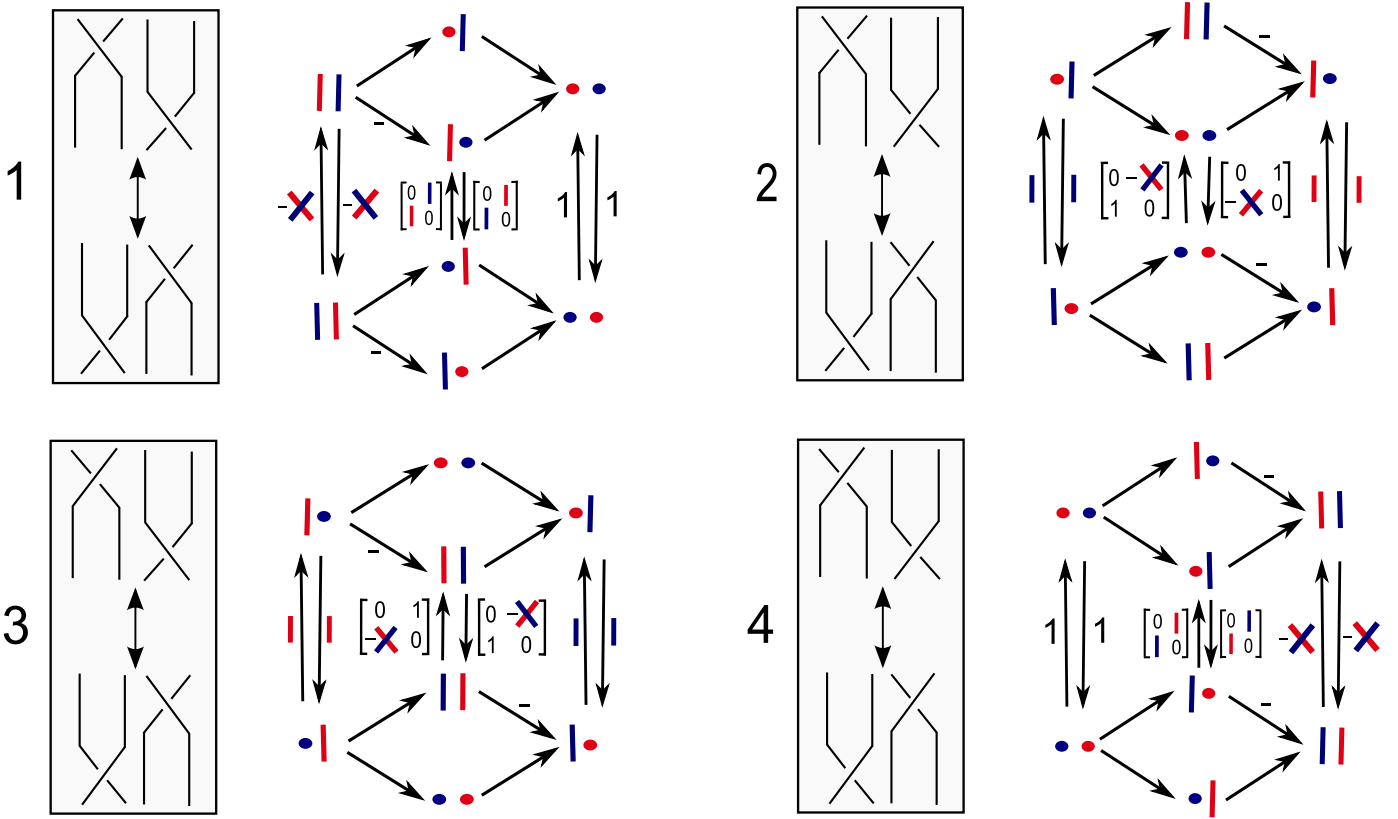}}
\caption{Slide generators} \label{MMgenSlides}
\end{figure}\

\vspace{10mm}

\item \textbf{Reidemeister $3$ generators} There are $12$ generators in all: 6 possibilities for the height orders of the 3 strands (denoted by a
number 1 through 6), and two directions for the movie (denoted "a" or "b"). Thankfully, the color-switching symmetries of the Soergel calculus allow
us explicitly list only $6$. The left-hand column lists the generators, and the chain complexes they correspond to; switching colors in the
complexes yields the corresponding generator listed on the right. Each of these variants has a free parameter $x$, and the parameter used
for each variant is actually independent from the other variants.

\begin{remark} Using sequences of R$2$-type generators and various movie moves we could have abstained from ever defining certain R$3$-type variants
or proving the movie moves that use them. We never use this fact, and list all here for completeness.\end{remark}

\begin{figure} [!htbp] \centerline{
\includegraphics[scale=.75]{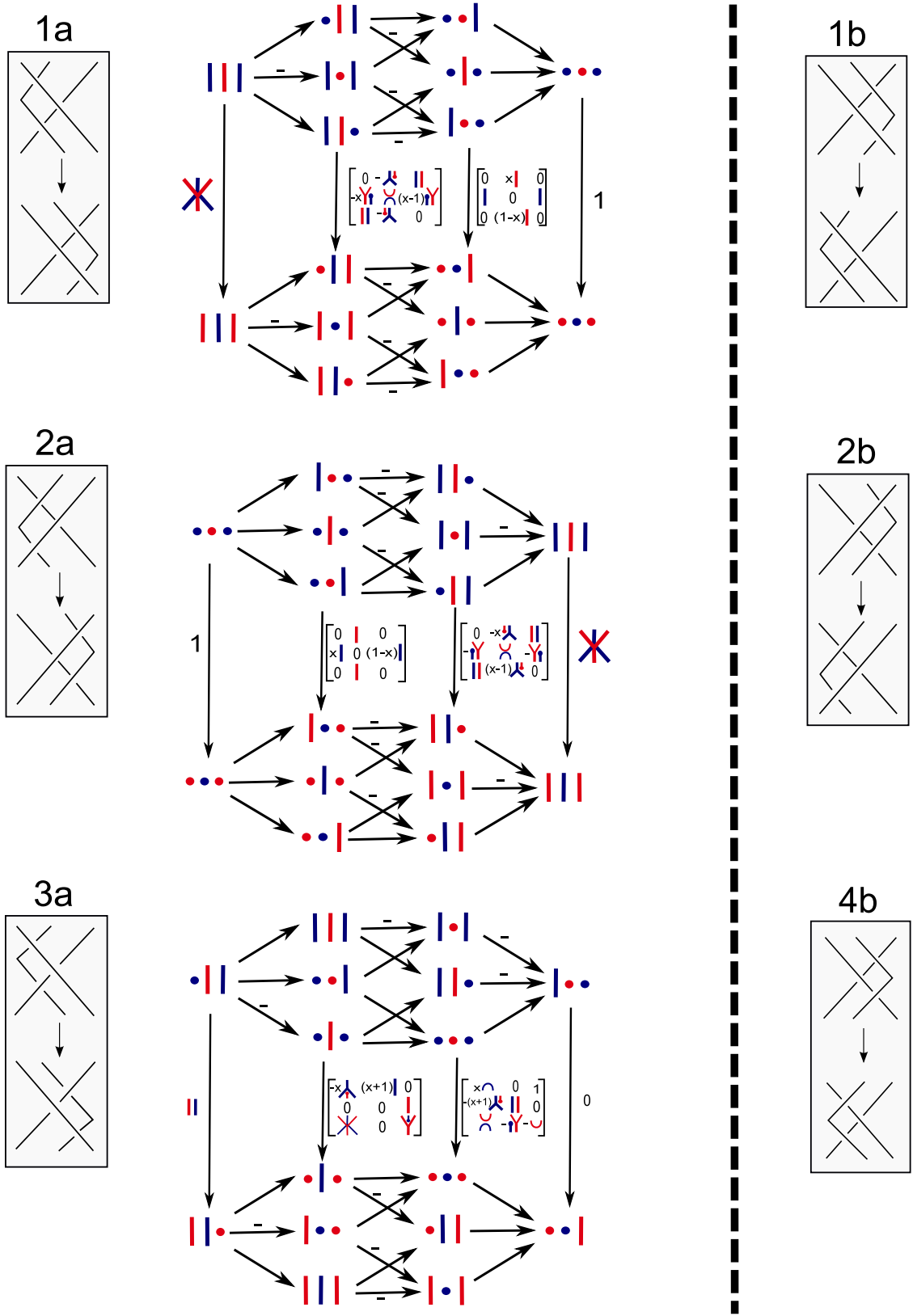}} \caption{Reidemeister $3$ type movie move generators}
\label{MMgenR3A} \end{figure}

\begin{figure} [!htbp]
\centerline{
\includegraphics[scale=.75]{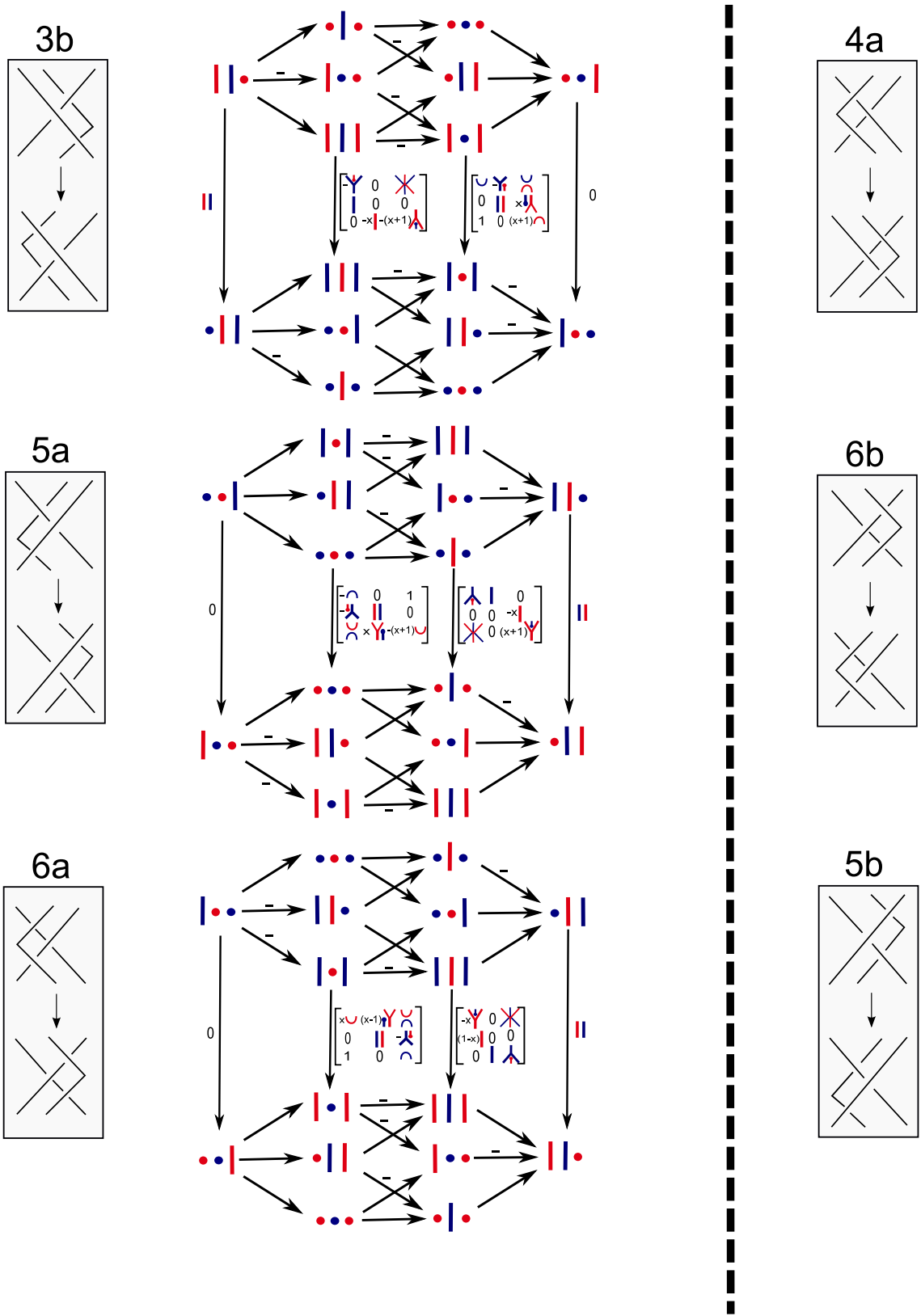}}
\caption{Reidemeister $3$ type movie move generators} \label{MMgenR3B}
\end{figure}

\end{itemize}

\pagebreak

\begin{claim} Up to homotopy, each of the maps above is independent of $x$.

\proof We prove the claim for generator 1a above; all the others follow from essentially the same computation. One can easily observe
that there are very few summands of the source complex which admit degree -1 maps to summands of the target complex. In fact, the unique (up to
scalar) non-zero map of homological degree -1 and graded degree -1 is a red trivalent vertex: a red fork which sends the single red line in the
second row of the source complex to the double red line in the second row of the target complex. Given two chain maps, one with free variable $x$
and one with say $x'$, the homotopy is given by the above fork map, with coefficient $(x-x')$. The homotopies for the other variants are exactly the
same, save for the position, color, and direction of the fork (there is always a unique map of homological and graded degree -1). \end{claim}

\begin{remark} For all movie generators, there is a summand of both the source and the target which is $B_{\emptyset}$. We have clearly used the
convention that for Type I movie generators, the induced map from the $B_{\emptyset}$ summand in the source to the $B_{\emptyset}$ summand in the
target is the identity map. It is true that, with this convention, the chain maps above are the unique chain maps which would satisfy the movie move
relations, where the only allowable freedom is given by the choice of various parameters $x$ (exercise). There is no choice up to
homotopy, so this is a unique solution. \end{remark}

\begin{remark} Ignoring this convention, each of the above maps may be multiplied by an invertible scalar. Some relations must be imposed between
these scalars, which the reader can determine easily by looking at the movie moves (each side must be multiplied by the same scalar). Movie move 11
forces all slide generators to have scalar $1$. Movie move 13 forces all R3 generators to have scalar $1$. Movie move 14 and 2 combined force the
scalar for any R2 generator to be $\pm 1$, and then movie moves 2 and 5 force this sign to be the same for all 4 variants. Movie move 12 shows that
the scalar for the birth of an overcrossing and the death of an undercrossing are related by the sign for the R2 generator. So the remaining freedom
in the definition of the functor is precisely a choice of one sign and one invertible scalar. \end{remark}

\section{Checking the Movie Moves}
\label{sec-moviemoves}

%
\subsection{Simplifications}
\label{subsec-simplify}
%

Given that the functor $F$ has been defined explicitly, checking that the movie moves hold up
to homotopy can be done explicitly.  One can write down the chain maps for both complexes, and either
check that they agree, or explicitly find the homotopy which gives the difference.  This is not
difficult, and many computations of this form were done as sanity checks. However, there
are so many variants of each movie move that writing down every one would take far too long.

Thanks to Morrison, Walker, and Clark \cite{CMW}, a significant amount of work
can be bypassed using a clever argument. The remainder of this section merely repeats results from that paper.

\begin{notation} Let $P,Q,T$ designate braid diagrams. $\Hom(P,Q)$ will designate the hom space between $F(P), F(Q)$ in
  the homotopy category of complexes in $\mc{SC}_2$. We write $\HOM$ for the graded vector space of
  all morphisms of complexes (not necessarily in degree 0).  $\Hom(B_{\ii},B_{\jj})$ will still
  designate the morphisms in $\mc{SC}_1$. Let $\1$ designate the crossingless braid diagram.
\end{notation}

\begin{lemma} (see \cite{CMW}) Suppose that Movie Move 2 holds. Then there is an adjunction
  isomorphism $\Hom(PO_i,Q) \to \Hom(P,QU_i)$, or more generally $\Hom(PT,Q) \to
  \Hom(P,Q\overline{T})$. Similarly for other variations: $\Hom(O_iP,Q) \to \Hom(P,U_iQ)$,
  $\Hom(P,QO_i) \to \Hom(PU_i,Q)$, etc.
\end{lemma}

\begin{proof} Given a map $f \in \Hom(PO_i,Q)$, we get a map in $\Hom(P,QU_i)$ as follows: take the
  R2 movie from $P$ to $PO_iU_i$, then apply $f \TenR \id_{U_i}$ to $QU_i$.  The reverse adjunction
  map is similar, and the proof that these compose to the identity is exactly Movie Move 2.
\end{proof}

\begin{cor} For any braid $P$, $\Hom(P,P) \cong \Hom(\1,P\overline{P})$. \end{cor}

Note that in the braid group, $P\overline{P} = \1$.

\begin{lemma} Suppose that Movie Moves 3, 5, 6, and 7 hold. Then if $P$ and $Q$ are two braid
  diagrams which are equal in the braid group, then $\Hom(P,T) \cong \Hom(Q,T)$. \end{lemma}

\begin{proof} If two braid diagrams are equal in the braid group, one may be obtained from the other
  by a sequence of R2, R3, and distant crossing switching moves. Put together, these movie moves
  imply that all of the above yield isomorphisms of complexes.  Thus $P$ and $Q$ have isomorphic
  complexes.
\end{proof}

\begin{remark} Technically, we don't even need these movie moves, only the resulting isomorphisms, which were already shown by Rouquier. However,
since these movie moves are easy to prove and we desired the proofs in this paper to be self-contained, we show the movie moves directly.
\end{remark}

Now the complex associated to $\1$ is just $B_{\emptyset}$ in homological degree 0.  So $\HOM(\1,\1)
= \HOM(B_{\emptyset},B_{\emptyset})$, which we have already calculated is the free polynomial ring
generated by double dots.  In particular, the degree 0 morphisms are just multiples of the identity.
Remember, this is a non-trivial fact in the graphical context! We will say more about this in Section \ref{subsec-brute}.

Putting it all together, we have

\begin{cor} Suppose that Movie Moves 2,3,5,6,7 all hold. If $P$ and $Q$ are braid diagrams which are
  equal in the braid group, then $\Hom(P,Q) \cong \Bbbk$, a one-dimensional vector space. \end{cor}

The practical use of finding one-dimensional Hom spaces is to apply the following method.

\begin{defn} (See \cite{CMW}) Consider two complexes $A$ and $B$ in an additive $\Bbbk$-linear
  category. We say that a summand of a term in $A$ is \emph{homotopically isolated} with respect to
  $B$ if, for every possible homotopy $h$ from $A$ to $B$, the map $dh+hd \colon A \to B$ is zero
  when restricted to that summand.
\end{defn}

\begin{lemma} Let $\phi$ and $\psi$ be two chain maps from $A$ to $B$, such that $\phi \equiv c\psi$
  up to homotopy, for some scalar $c \in \Bbbk$. Let $X$ be a homotopically isolated summand of
  $A$. Then the scalar $c$ is determined on $X$, that is, $\phi = c \psi$ on $X$. \end{lemma}

The proof is trivial, see \cite{CMW}.  The final result of this argument is the following corollary.

\begin{cor} Suppose that Movie Moves 2,3,5,6,7 all hold. If $P$ and $Q$ are braid diagrams which are
  equal in the braid group, and $\phi$ and $\psi$ are two chain maps in $\Hom(P,Q)$ which agree on a
  homotopically isolated summand of $P$, then $\phi$ and $\psi$ are homotopic.
\end{cor}

\begin{proof} Because the Hom space modulo homotopy is one-dimensional, we know there exists a
  constant $c$ such that $\phi \equiv c \psi$. The agreement on the isolated summand implies that
  $c=1$. \end{proof}

Most of the movie generators are isomorphisms of complexes; only birth and death are not.  Hence,
Movie Moves 1 through 10 all consist of morphisms $P$ to $Q$, for $P$ and $Q$ equal in the braid
group.  Finding a homotopically isolated summand and checking the map on that summand alone will greatly
reduce any work that needs to be done.  Of course, one must show Movie Moves 2,3,5,6,7 independently
before this method can be used.

One final simplification, also found in Morrison, Walker and Clark, is that modulo Movie Move 8 all
variants of Movie Move 10 are equivalent. Hence we can prove Movie Move 10 by investigating solely the
overcrossing-only variant.

These simplifications apply to any functorial theory of braid cobordisms, so long as $\Hom(\1,\1)$
is one-dimensional.  Now we look at what we can say specifically about homotopically isolated
summands for Rouquier complexes in $\mc{SC}_2$.

Any homotopy must be a map of degree -1 (if we ignore degree shifts on objects, as in our
conventions).  There are very few maps of negative degree in $\mc{SC}_2$, a fact which immediately
forces most homotopies to be zero.  For instance, there are no negative degree maps from
$B_{\emptyset}$ to $B_i$, for any $i$.  In an overcrossing-only braid, where
$B_{\emptyset}$ occurs in the maximal homological grading and various $B_i$ show up in the
penultimate homological grading, the $B_{\emptyset}$ summand is homotopically isolated! Thus the
overcrossing-only variant of Movie Move 10 will be easy.  In fact, because of the convention we use
that all isomorphism movie generators will restrict to multiplication by 1 from the $B_{\emptyset}$
summand to the $B_{\emptyset}$ summand, checking Movie Move 10 is immediate.

The only generators of negative degree are trivalent vertices. If each color appears no more than
once in a complex, then there can be no trivalent vertices, so no homotopies are possible. This will
apply to every variant of Movie Move 4, for instance.

Deducing possible homotopies is easy, as there are very few possibilities. For instance, the only nonzero maps which occur in homotopies outside
of Movie Move 10 are:

\vspace{3mm}

\begin{figure} [!htbp]
\centerline{
\includegraphics[scale=.85]{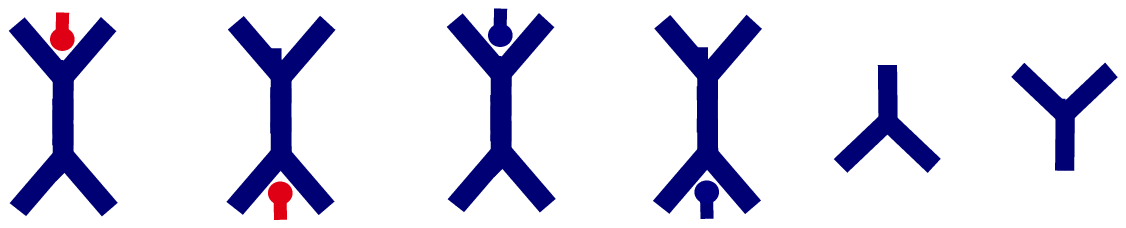}}
\label{homotopymaps}
\end{figure}\

\vspace{-1mm}

We will not use these simplifications to their maximal effect, since some checks are easy enough to do
without. For a discussion of other implications of
checking the movie moves by hand, see Section \ref{subsec-brute}.

%
\subsection{Movie Moves}
\label{subsec-moviemoves}
%

\note \textbf{(Logical sequence in the proofs of the movie moves.)} We list the movie moves in numerical order, as opposed to logical order of
interdependence. To use the technical lemma about homotopically isolated summands we first need to check movie moves 2,3,5,6,7. The reader will see
that we prove these through direct computation, relying on none of the other moves.


\begin{itemize} \item \textbf{MM1} There are eight variants of this movie (sixteen if you count the horizontal flip, which is just a color symmetry), of which we present two explicitly here. The key fact is that every slide
generator behaves the same way: chain maps on summands have either a color crossing with a minus sign, the identity map with a plus sign, or zero;
these maps occur precisely between the only summands where they make sense and, hence, have the same signs on both sides of the movie. Reversing
direction uniformly changes the sign on the cups or caps in the R2 move. The only interesting part of the check uses a twist of relation
(\ref{twist4}). We describe in detail the movie associated to the first generator in figure \ref{MM1a}, and give the composition associated to
generator $3$ in figure \ref{MM1b}. Note that this check is trivial anyway since every summand is homotopically isolated.

\begin{figure}[!htbp] \centerline{ \includegraphics[scale=.8]{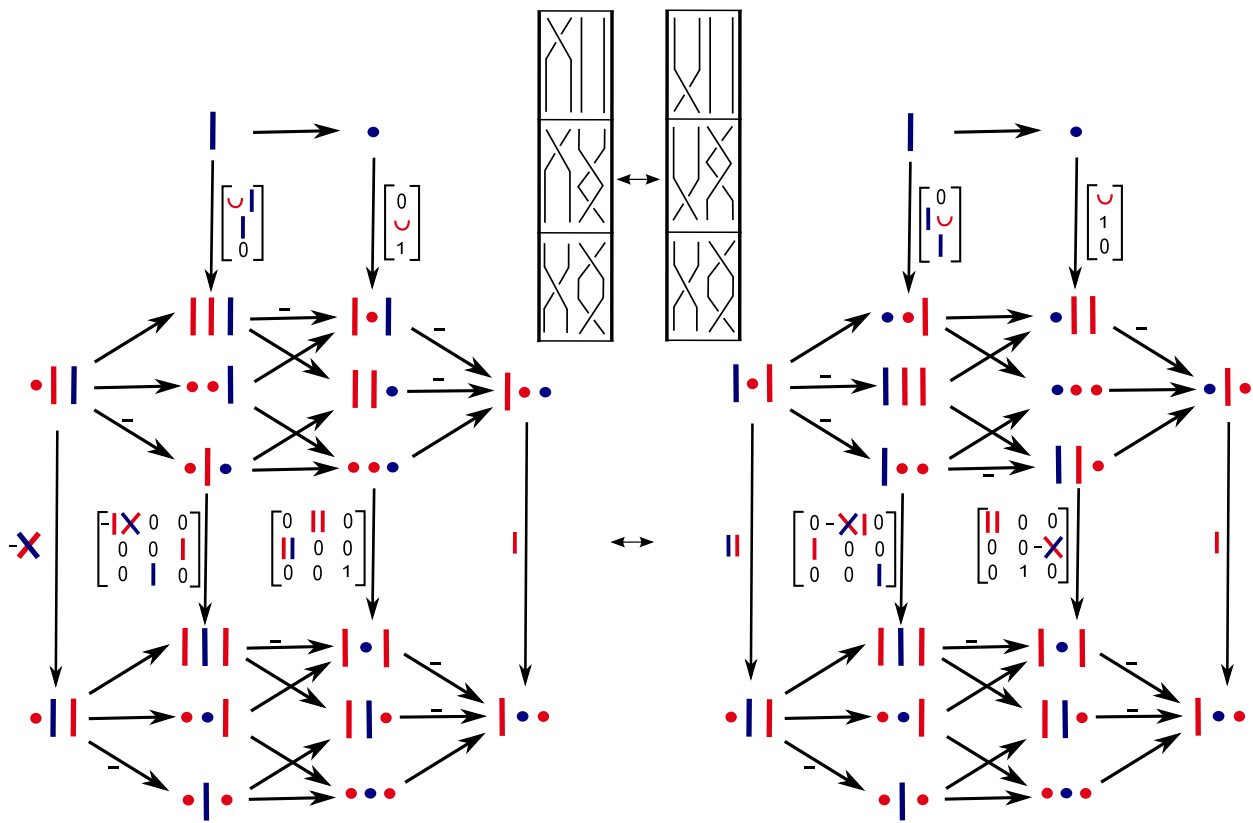}} \caption{Movie Move $1$
associated to slide generator $1$} \label{MM1a} \end{figure}

\pagebreak

 \begin{figure}[!htbp] \centerline{ \includegraphics[scale=.8]{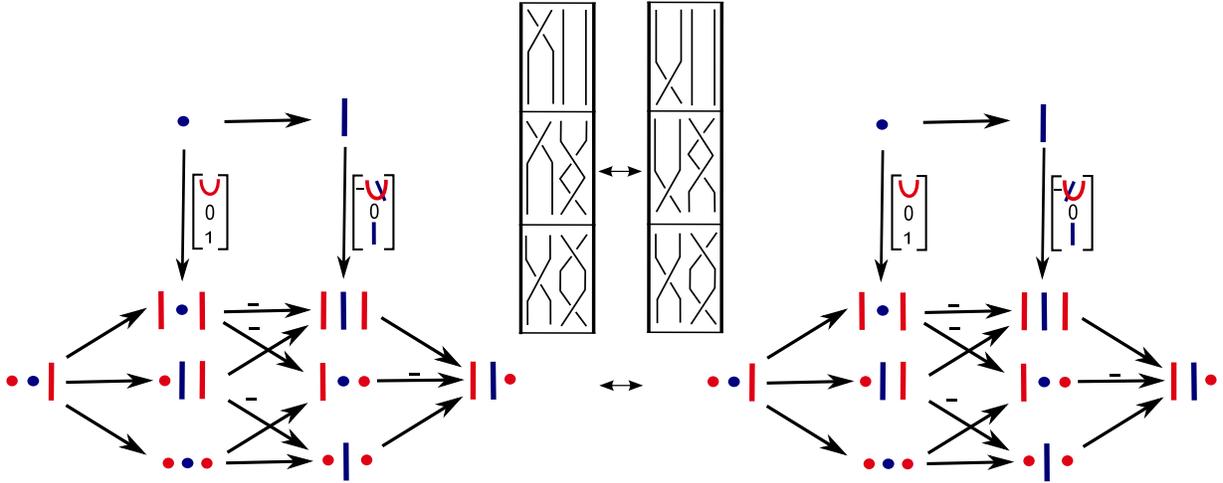}} \caption{Movie
Move $1$ associated to slide generator $3$} \label{MM1b} \end{figure}

\item \textbf{MM2} There are $4$ variants to deal with here; we describe only one, and similar reasoning to that of MM$1$ will convince the reader that the other $3$ are readily verified. The composition has the following form: 
\begin{figure}[!htbp] 
\centerline{
\includegraphics[scale=.8]{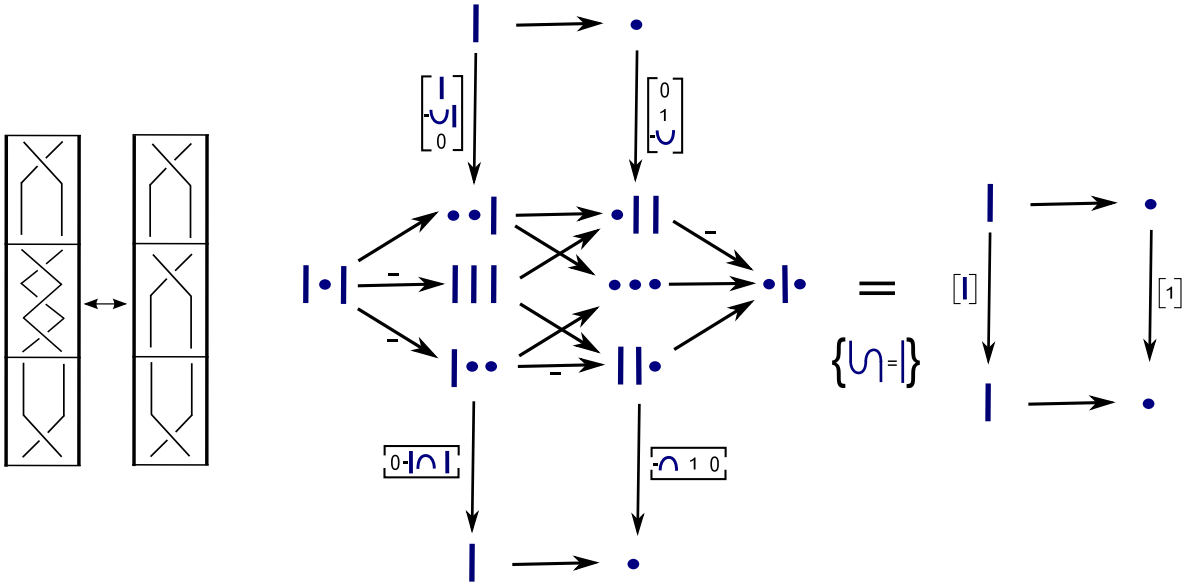}}
\caption{Movie Move $2$} \label{MM2}
\end{figure}

\pagebreak

\item \textbf{MM3} All $8$ movie move $3$ variants are essentially immediate after glancing at the slide generators, but we list one for posterity:
\begin{figure}[!htbp] 
\centerline{
\includegraphics[scale=.9]{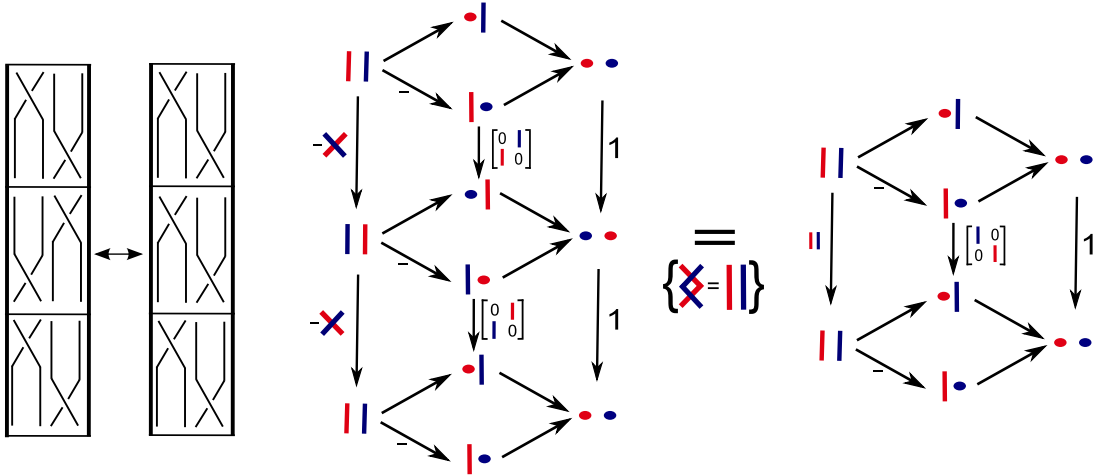}}
\caption{Movie Move $3$} \label{MM3}
\end{figure}


\item \textbf{MM4} At this point the conscientious reader will find all $16$ variants of movie move $4$ quite easy, for the regularity of the slide chain maps allows one to write the compositions for the left and right-hand side at once. The maps only differ at the triple-color crossings, so we have to make use of relation (\ref{distslide4}). 
\begin{figure}[!htbp] 
\centerline{
\includegraphics[scale=.8]{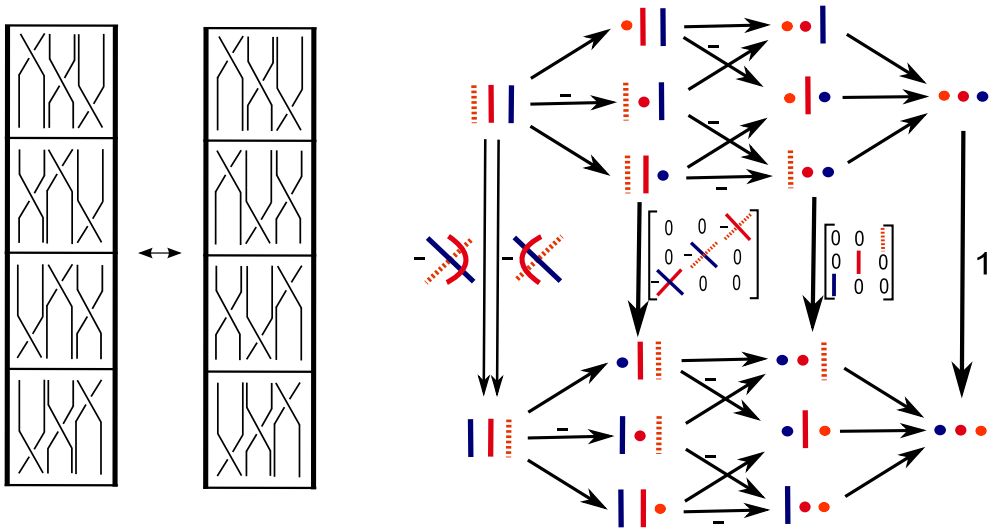}}
\caption{Movie Move $4$} \label{MM4}
\end{figure}

\pagebreak

\item \textbf{MM5} There are two variants of this movie, with the calculation for both almost identical. We consider the move associated to the first generator. The compostion has the following form: 

\vspace{4mm}

\begin{figure}[!htbp] 
\centerline{
\includegraphics[scale=.8]{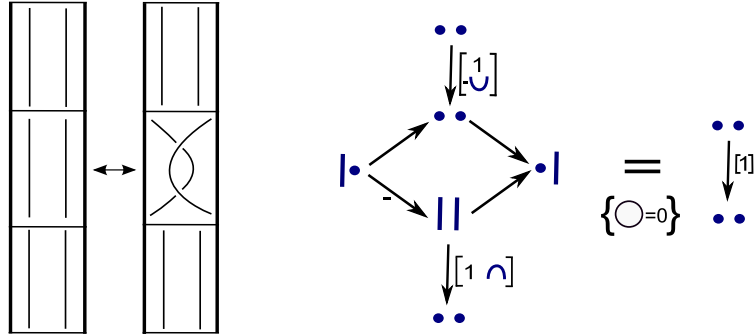}}
\caption{Movie Move $5$} \label{MM5}
\end{figure}

\vspace{7mm}

\item \textbf{MM6} Again there are two variants and the calculation is almost as easy as the one for MM5; the only difference is that here we
actually have to produce a homotopy. We check the variant associated to generator $1$; left arrows are the identity, right the composition, and
dashed the homotopy. Checking that the homotopy works requires playing with relation (\ref{dotslidesame}).

\begin{figure} [!htbp]
\centerline{
\includegraphics[scale=.8]{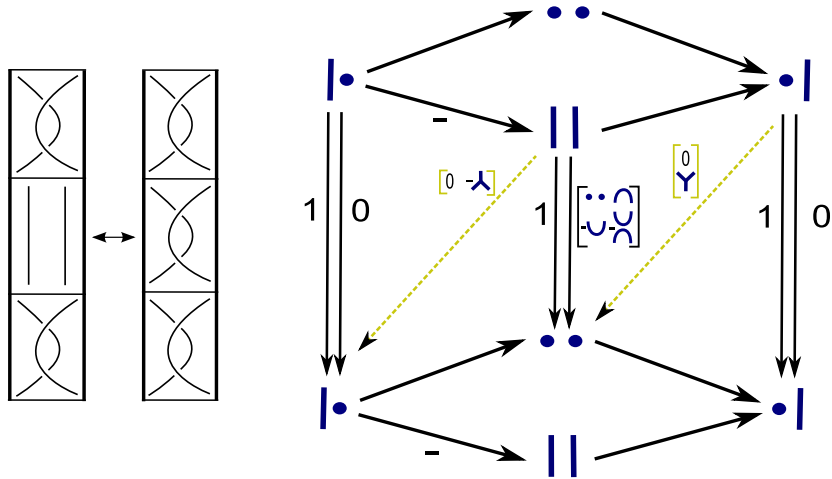}}
\caption{Movie Move $6$} \label{MM6}
\end{figure}

\pagebreak

\item \textbf{MM7} There are $12$ variants of MM$7$, one for each R3 generator, and color symmetry will immediately reduce the number of different checks to
$6$; nevertheless, this is still a bit a drudge as each one requires a homotopy and a minor exercise in the relations. We display the movie associated to
generator $1$a and leave it to the very determined reader to repeat a very similar computation the remaining $5$ times. The chain maps for the left-hand side
of the movie are the following:

\vspace{7mm}

\begin{figure} [!htbp]
\centerline{
\includegraphics[scale=1]{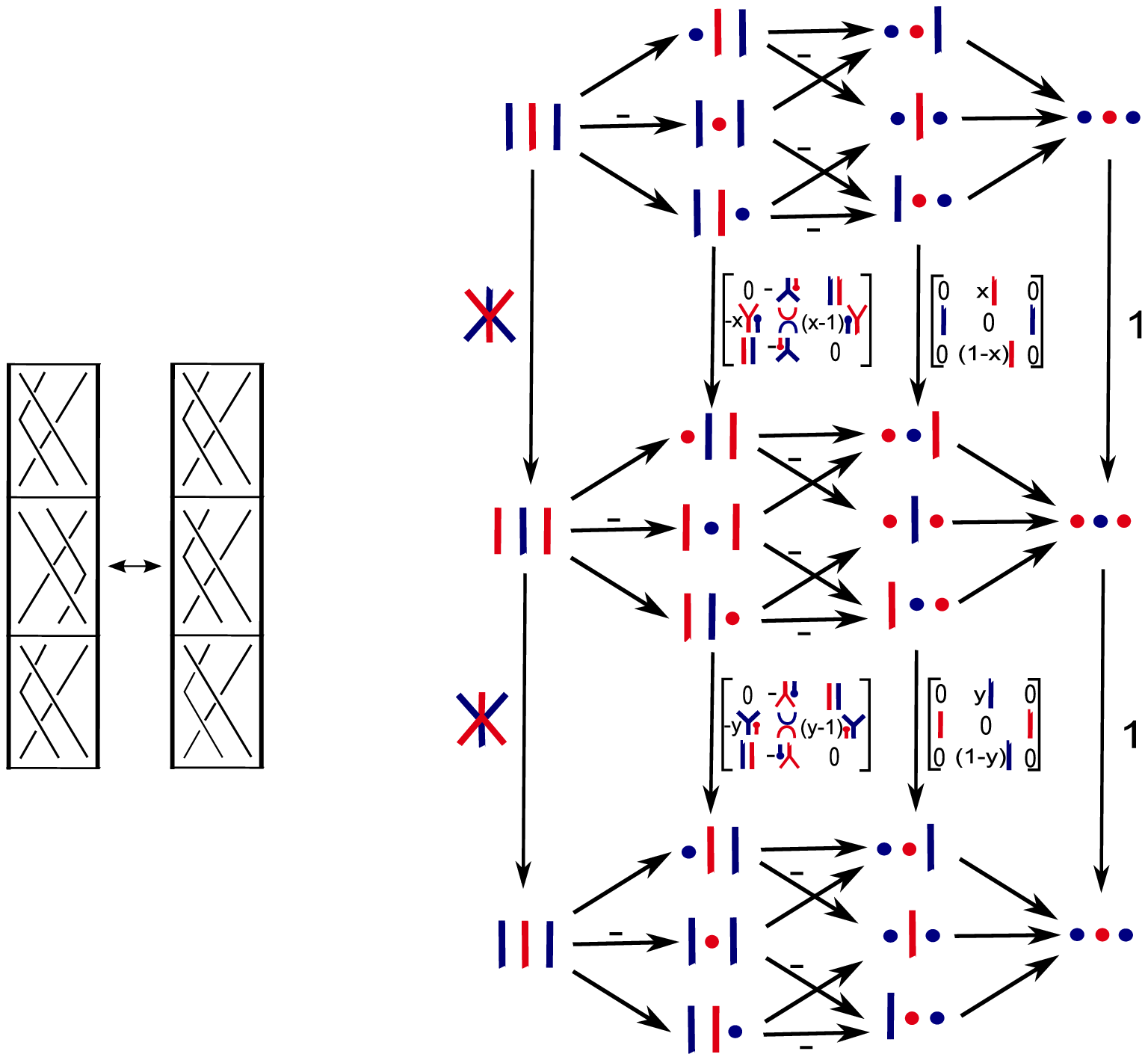}}
\caption{Movie Move $7$} \label{MM7a}
\end{figure}

\pagebreak 

The composition and homotopy is:

\begin{figure} [!htbp]
\centerline{
\includegraphics[scale=.8]{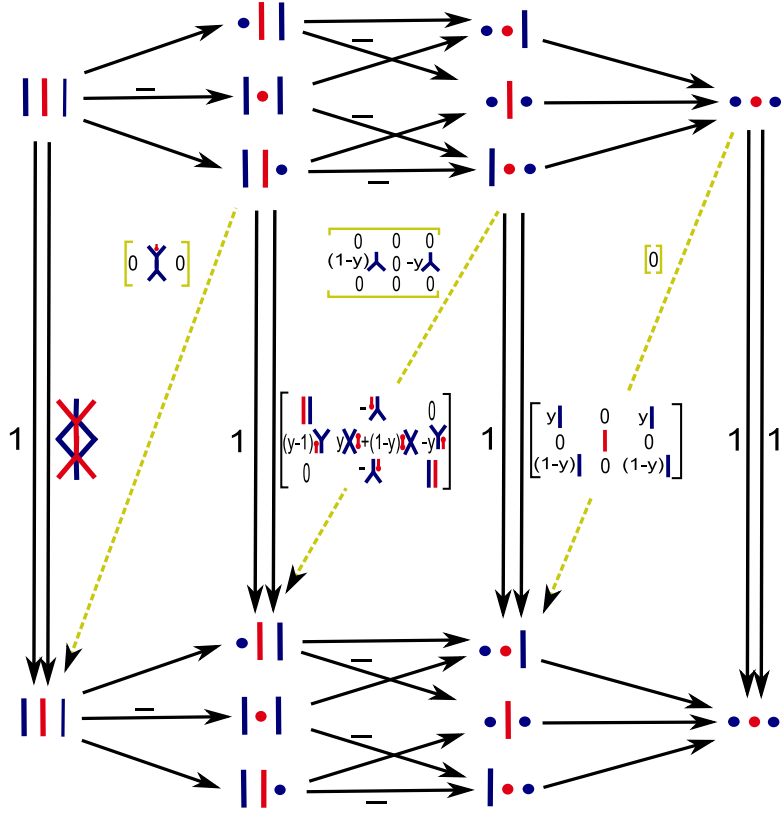}}
\caption{Homotopy for Movie Move $7$} \label{MM7b}
\end{figure}

To check that the prescribed maps actually give a homotopy between and composition and the identity still requires some manipulation. The
verification for the left-most map is simply relation (\ref{ipidecomp}). The verification for the right-most map is immediate, and for the third
map is simple. This leaves us with the second map. Here $dH + Hd=$

\begin{figure} [!htbp]
\centerline{
\includegraphics[scale=1]{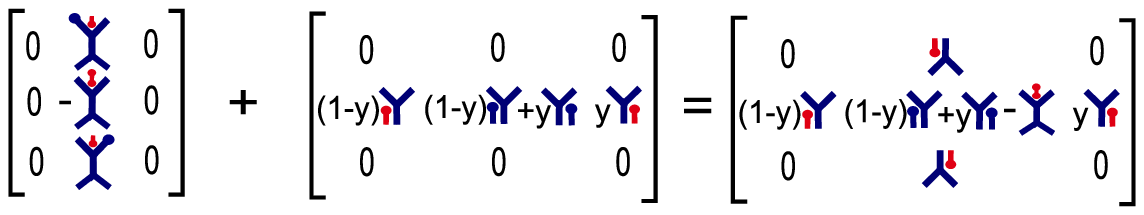}}
\label{MM7c}
\end{figure}

which save for the central entry is precisely the identity minus the composition. Equality of the central entry follows from this computation:

\begin{figure} [!htbp]
\centerline{
\includegraphics[scale=1]{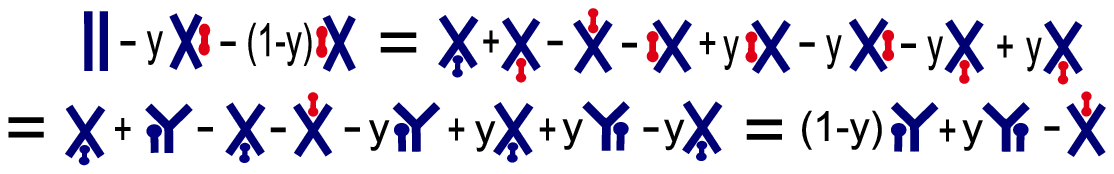}}
\label{MM7d}
\end{figure}

\vspace{-7mm}

\note This computation was done using relation (\ref{dotslidenear}) numerous times.


\item \textbf{MM8} There are twelve variants of MM$8$: $3!$ possibilities for height order, and two directions the movie can run. All twelve are
dealt with by the same argument, using a homotopically isolated summand. There are no degree $-1$ maps from $B_{\emptyset}$ to any summand in the
target, since there are at most two lines of a given color in the target, so we can assume there are no trivalent vertices. Hence the
$B_{\emptyset}$ summand of the source is homotopically isolated, so we need only keep track of the homological degree $0$
part, which significantly simplifies the calculation. We present one variant in diagram \ref{MM8}.

\begin{figure} [!htbp]
\centerline{
\includegraphics[scale=.8]{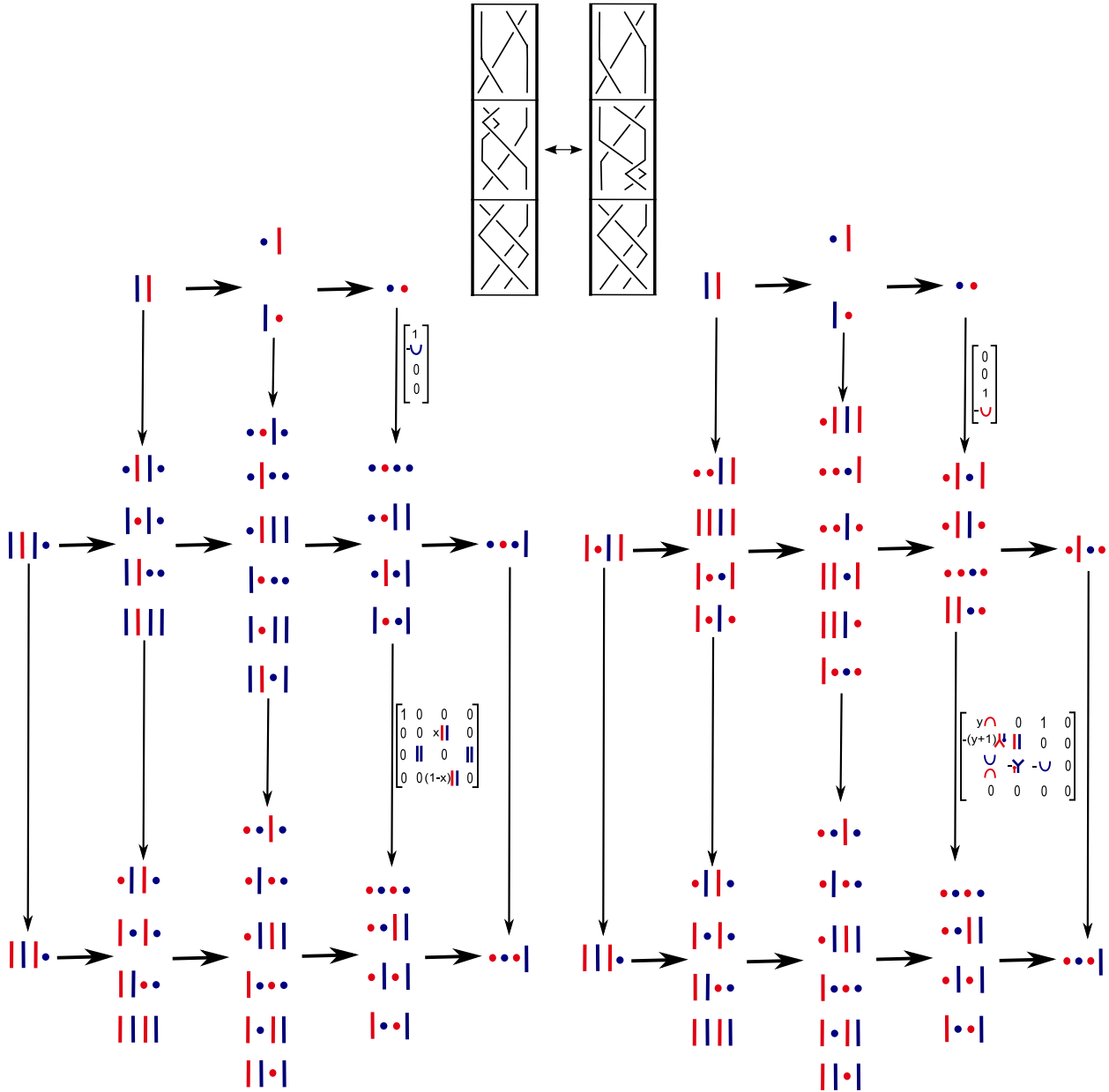}}
\caption{Movie Move $8$} \label{MM8}
\end{figure}

Composing the chain maps for the two sides of MM$8$ we see that they agree on $B_{\emptyset}$.

\pagebreak

\item \textbf{MM9} There are a frightful $96$ versions of MM$9$, coming from all the different R$3$ moves that can be done ($12$ in all), the
type of crossing that appears in the slide, and horizontal and vertical flips. Once again, homotopically isolated summands come to the rescue.
Again, in each variant there are no more than two crossings of a given color, so all maps from $B_{\emptyset}$ to each summand in the target have
non-negative degree. Thus the $B_{\emptyset}$ summand of the source is homotopically isolated. Three colors are involved, the distant color and two
adjacent colors. In the $B_{\emptyset}$ summand, the distant-colored line does not appear, and no application of a distant slide or R3 move can
make it appear. When the distant-colored line does not appear, the distant slide move acts by the identity. Thus both the right and left sides of
the movie act the same way on the $B_{\emptyset}$ summand, namely, they perform the R3 operation to it (sending it to the appropriate summands of
the target).

\begin{figure} [!htbp]
\centerline{
\includegraphics[scale=.7]{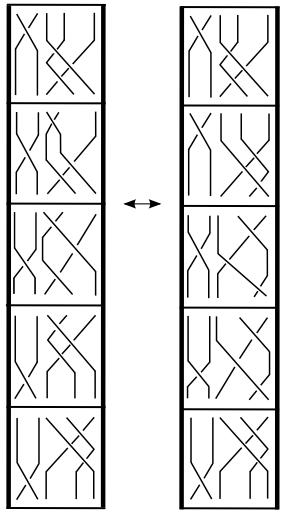}}
\caption{Movie Move $9$} \label{MM9}
\end{figure}

\item \textbf{MM10} The sheer burden of writing down the complexes and calculating the chain maps for even one version of MM$10$ is best avoided
at all costs. Despite at first seeming the more complicated of the movie moves, it is in the end the easiest to verify. We begin noting that,
once one has shown MM$8$, all of the versions of MM$10$ are equivalent (see section 3.2.2 in \cite{CMW}). So let us consider the variant with all left crossings.
We see immediately that the $B_{\emptyset}$ summand is homotopically isolated, that it is the unique summand in homological degree 0 in every
intermediate complex, and that the chain maps all act by the identity in homological degree 0. Hence both sides agree on a homotopically isolated
summand.

\pagebreak

\item \textbf{MM11} There are $32$ variants of MM$11$: 2 choices of crossing, a vertical and a horizontal flip, and the direction of the movie.
Half of these have chain maps that compose to zero on both sides, since the birth of a right crossing or the death of a left crossing is the zero
chain map. The rest are straighforward. We give an example below in figure \ref{MM11}.

\vspace{5mm}

\begin{figure} [!htbp]
\centerline{
\includegraphics[scale=.9]{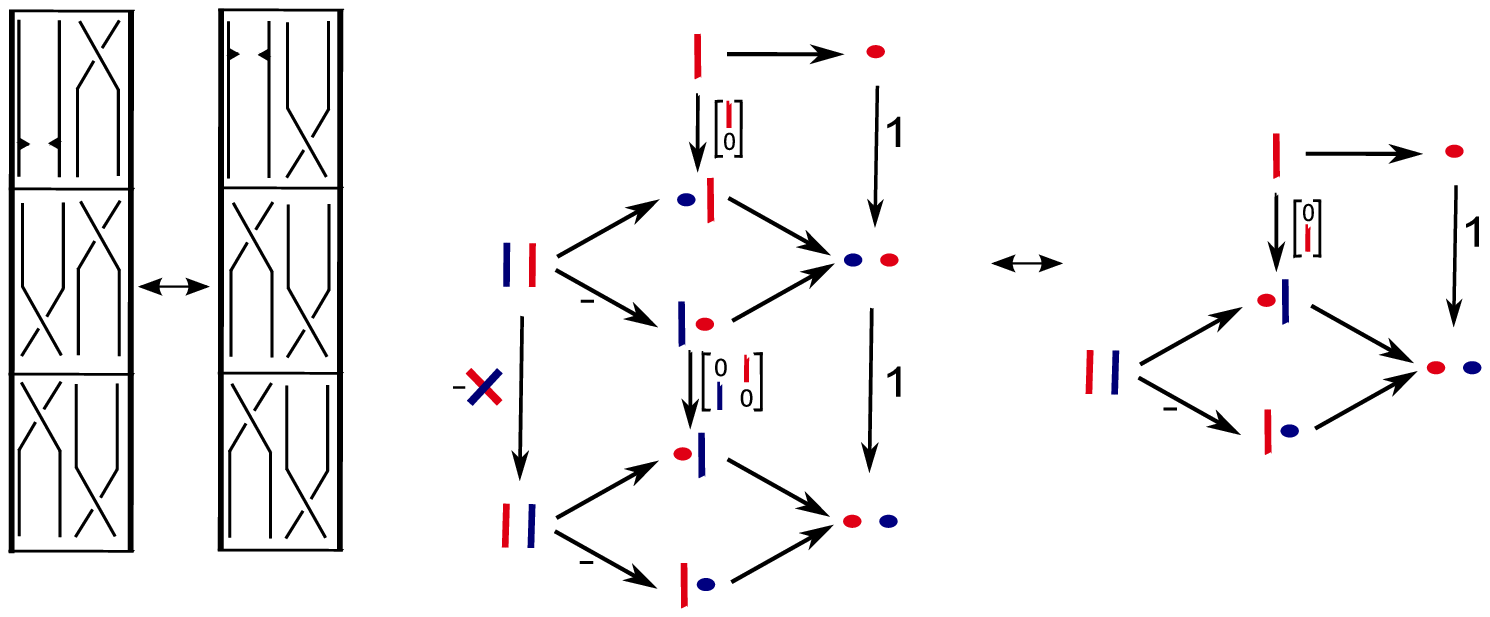}}
\caption{Movie Move $11$} \label{MM11}
\end{figure}

\vspace{5mm}

\item \textbf{MM12} There are $8$ variants: a choice of R2 move, a vertical flip, and the direction of the movie. Again, half of these are zero all around.  Here are two variants; the other two are extremely similar.

\vspace{5mm}

\begin{figure} [!htbp]
\centerline{
\includegraphics[scale=.9]{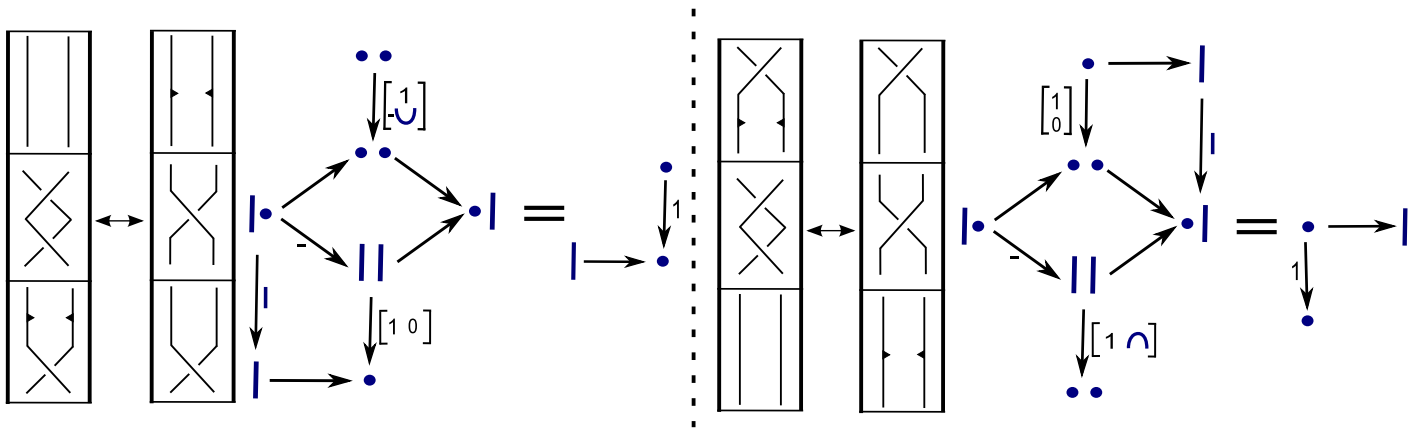}}
\caption{Movie Move $12$} \label{MM12}
\end{figure}

\pagebreak

\item \textbf{MM13} There are $24$ variants: $12$ R$3$ generators and two directions. Half are zero, and color symmetry for R3 generators reduces
the number to check by half again. For the 6 remaining variants, the check requires little more than just writing down the composition, since the
required homotopy in each instance is quite easy to guess. In figure \ref{MM13} we describe the variant associated to to the first R$3$
generator.

\begin{figure} [!htbp]
\centerline{
\includegraphics[scale=.9]{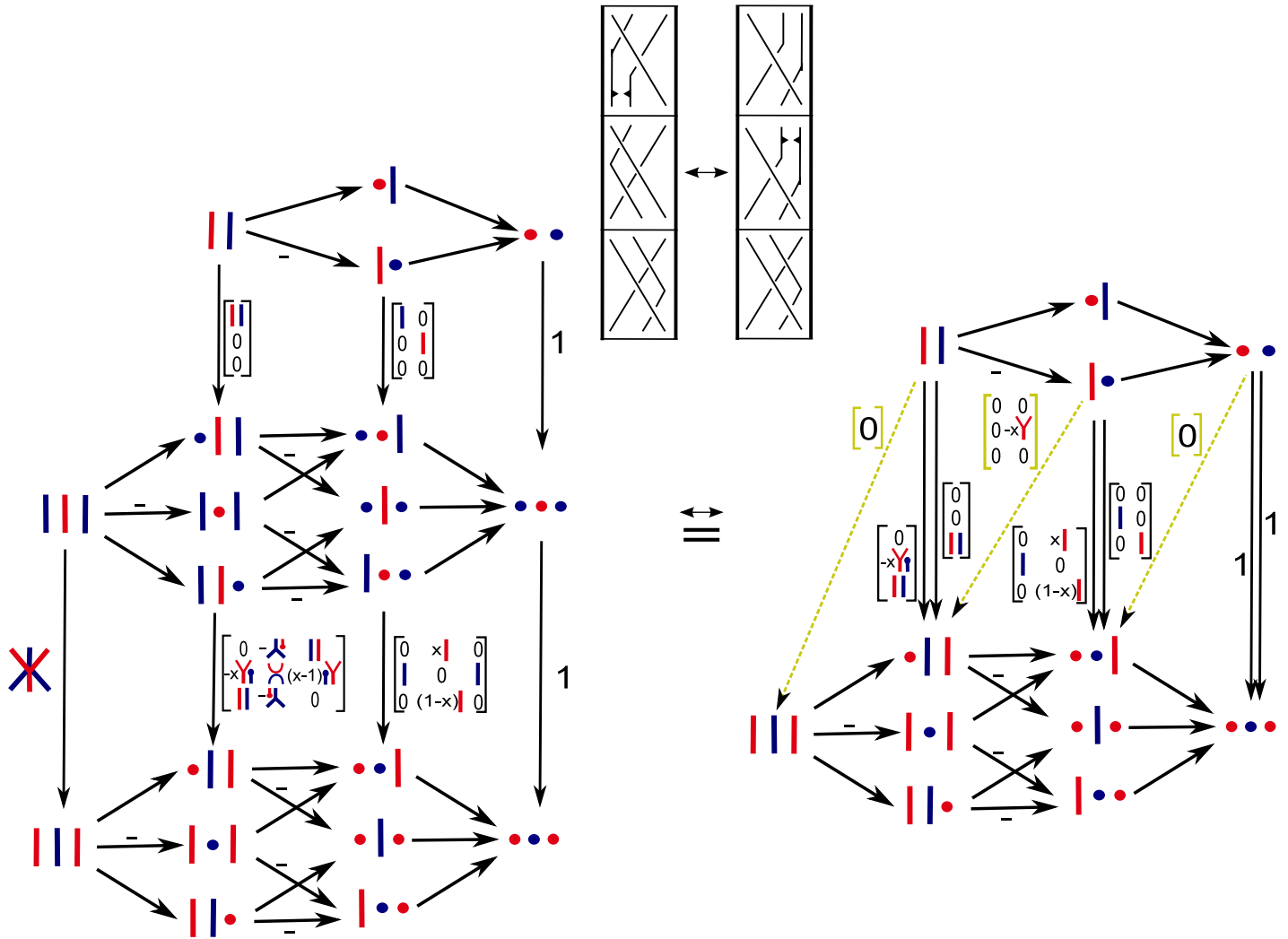}}
\caption{Movie Move $13$} \label{MM13}
\end{figure}

\pagebreak

\item \textbf{MM14} Since none of the R$3$ generators of type $1$ or $2$ is compatible with MM$14$, we are left with $16$ variants: 8 R3
generators and 2 directions. As usual, there are only $4$ to check. In addition to this, the initial frame of the movie corresponds to a complex
supported in homological degree $0$ only, so we only need write down what happens there. In figure \ref{MM14} we describe the variant associated
to the R$3$ generator $3$a.

\begin{figure} [!htbp]
\centerline{
\includegraphics[scale=.8]{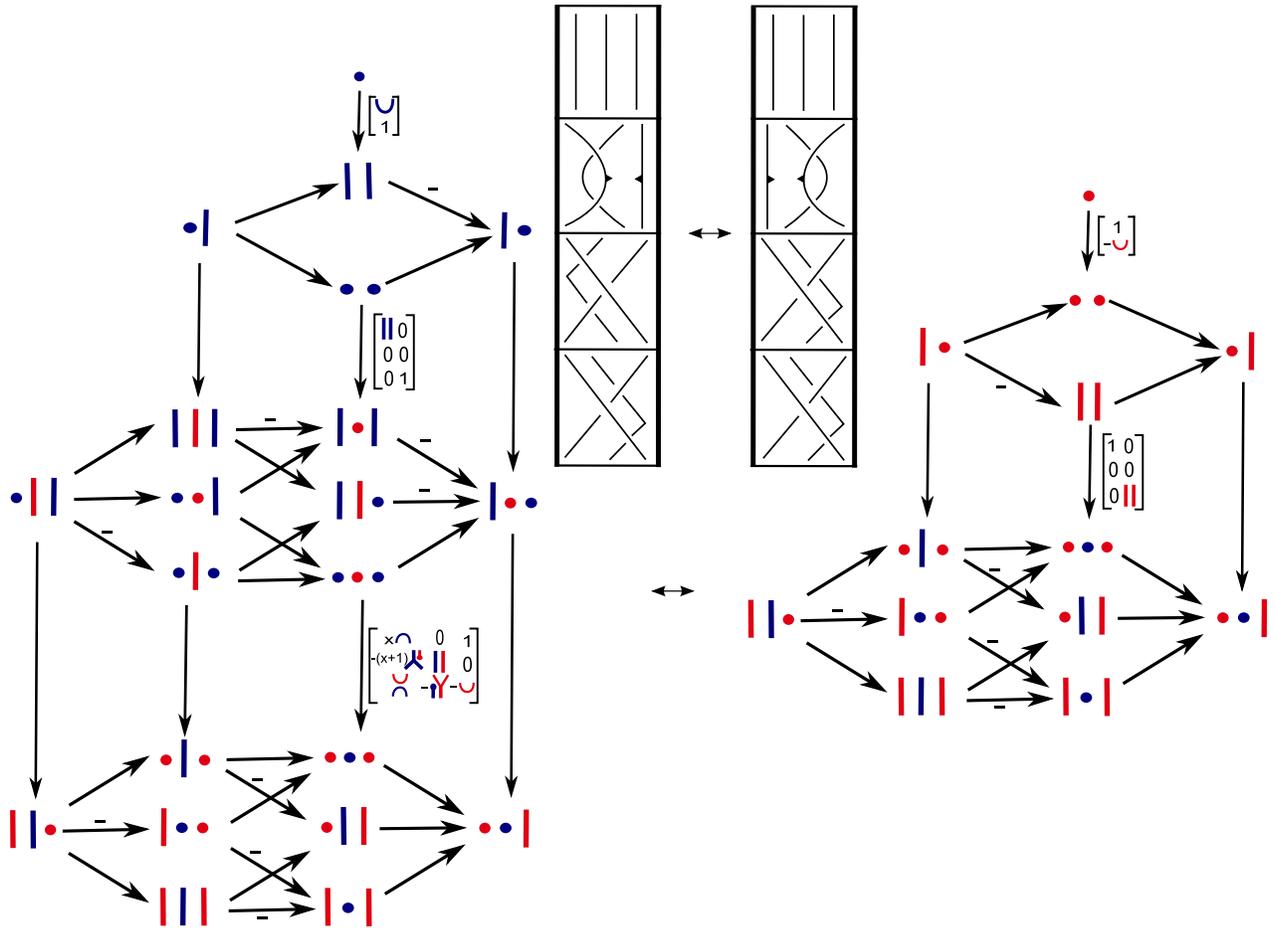}}
\caption{Movie Move $14$} \label{MM14}
\end{figure}

\end{itemize}

\pagebreak

\section{Additional Comments}
\label{sec-additional}

%
\subsection{The Benefits of Brute Force}
\label{subsec-brute}
%

We have now shown that there is a functor from the braid cobordism category into the homotopy category of complexes in $\mc{SC}_2$. Our method of
proof used homotopically isolated summands, and hence relied on the fact that $\Hom(B_{\emptyset},B_{\emptyset})$ was 1-dimensional. This is a
trivial fact in the context of $R$-bimodules, amounting to the statement that $\HOM(R,R)=R$. However, it is a non-trivial fact to prove for the
graphical definition of $\mc{SC}_1$, requiring the more complicated graphical proofs in \cite{EKh}. Moreover, $\Hom(B_{\emptyset},B_{\emptyset})$
need not be 1-dimensional in some arbitrary category $\mc{C}$ of which $\mc{SC}_1$ is a (non-full) subcategory, and we may be interested in such
categories $\mc{C}$. For instance, it would be interesting to define such a category $\mc{C}$ for which one would have all  birth
and death maps nontrivial (although the authors have yet to find an \emph{interesting} extension of this type).

Our method of proof, however, is irrelevant and the truth of Theorem \ref{mainthm} does not depend on it. One could avoid any machinery by
checking each movie move explicitly (in fact, the only ones that remain to be checked are MM8, MM9, and MM10). Checking even a single variant of
MM10 by brute force is extremely tedious, since each complex has 64 summands, but it could be done. In addition, we have actually proven slightly more: for
any additive monoidal category $\mc{C}$ having objects $B_i$ and morphisms satisfying the $\mc{SC}_1$ relations, we can define a functor from the
braid cobordism category into the homotopy category of complexes in $\mc{C}$. This is an obvious corollary, since that same data gives a functor
from $\mc{SC}_2$ to $\mc{C}$. If one chose to change the birth and death maps, the proof for movie moves 1 through 10 would be unchanged, and one
would only need to check 11 through 14.

One other benefit to (theoretically) checking everything by hand is in knowing precisely which coefficients are required, and thus understanding the
dependence on the base ring $\Bbbk$. In all the movie moves we check in this paper, each differential, chain map, and homotopy has integral
coefficients (or free variables which may be chosen to be integral). In fact, every nonzero coefficient that didn't involve a free variable was
$\pm 1$, and free variables may be chosen such that every coefficient is 1,0, or $-1$. From our other calculations, the same should be true for
MM8 through MM10 as well (Khovanov and Thomas \cite{KT} already showed that Rouquier complexes lift over $\Z$ to a projective functor, which implies the
existence of homotopy maps over $\Z$). The next section discusses the definition of this functor in a $\Z$-linear category.

As an additional bonus, checking the movie moves does provide some intuition as to why $\mc{SC}_1$ has the relations that it does. One might wonder why these
particular relations should be correct: in \cite{EKh} we know they are correct because they hold in the $R$-bimodule category and because they are sufficient
to reduce all graphs to a simple form. There should be a more intuitive explanation.

As an illustrative example, consider the overcrossing-only variation of Movie Move 10 and the unique summand of lowest (leftmost)
homological degree: it is a sequence of 6 lines. Then the left hand movie and the right hand movie correspond to the following maps on this
summand:

\igc{.7}{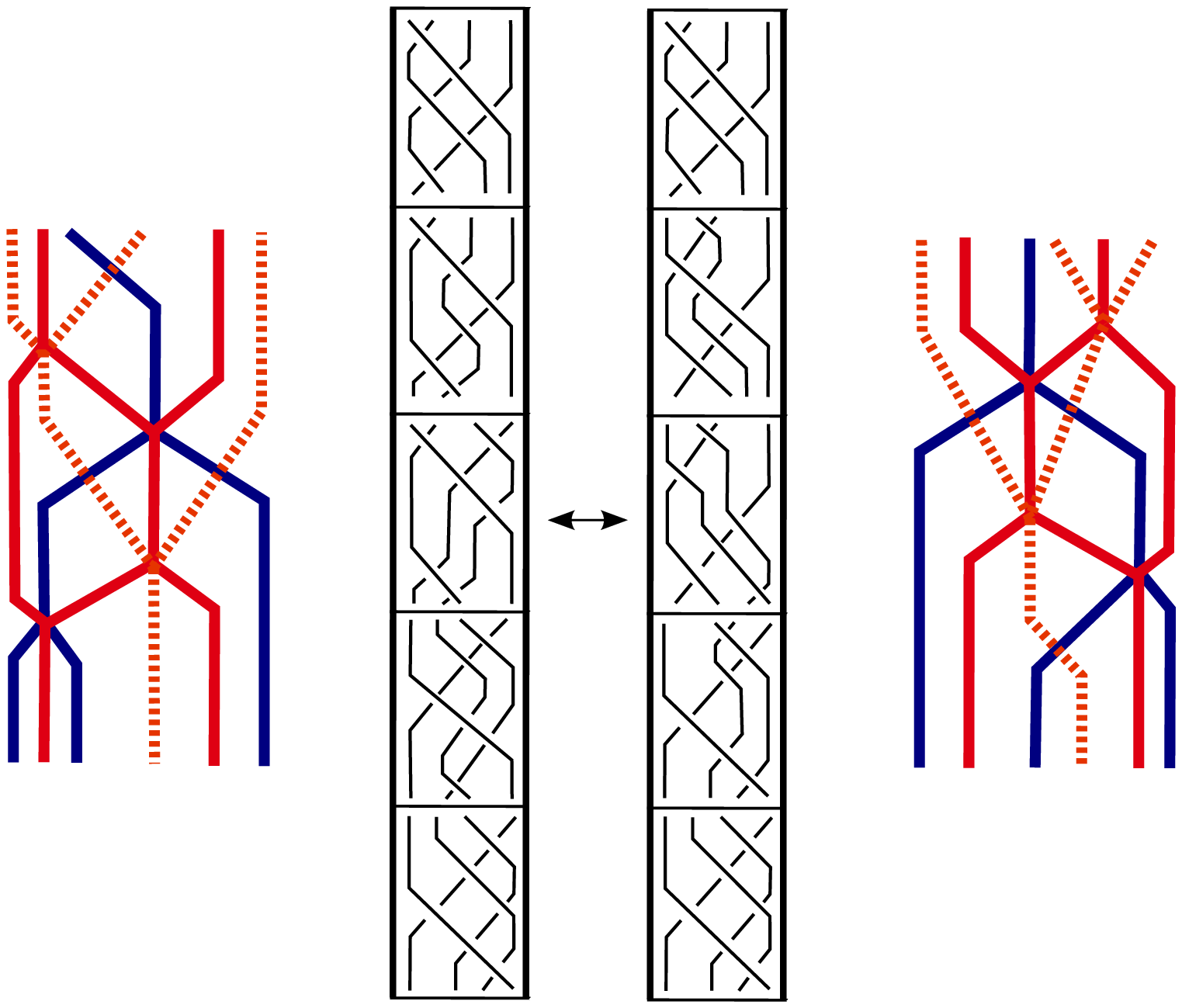}

Thus equality of these two movies on the highest term, modulo relation (\ref{R2}), is exactly relation (\ref{assoc3}).

Similarly, the highest terms in various other movie move variants utilize the other relations, as in the chart below.

\begin{center}
\begin{tabular}{|l|r|} \hline
MM & Relation \\ \hline 
1 & (\ref{twist4}) \\ 
2 & (\ref{twistline}) \\ 
3 & (\ref{R2}) \\ 
4 & (\ref{distslide4}) \\ 
5 & (\ref{needle}) \\
8 & (\ref{twist6}) \\
9 & (\ref{distslide6}) \\
10 & (\ref{assoc3}) \\   \hline 
\end{tabular}
\end{center}

We can view these relations heuristically as planar holograms encoding the equality of cobordisms given by the movie moves.

More relations are used to imply that certain maps are chain maps, or that homotopies work out correctly. For example, relation
(\ref{distslidedot}) is needed for the slide generator to be a chain map. One can go even further, although we shall be purposely vague: so long
as one disallows certain possibilities (like degree $\le 0$ maps from a red line to a blue line, or negative degree endomorphisms of
indecomposable objects) then our graphical generators must exist a priori, and must satisfy a large number of the relations above.

Type II movie moves (11 through 14) do not contribute any relations or requirements not already forced by Type I movie moves (although they do
fix the sign of various generators).

Almost every relation in the calculus is used in a brute force check of functoriality (including the brute force checks of MM8-10). However,
there are two exceptions: (\ref{assoc1}) and (\ref{assoc2}). Both these relations are in degree -2, and degree -2 does not appear in chain maps or
homotopies, so they could not have appeared. Nonetheless, they are effectively implied by the remainder of the relations. It is not hard to use
the rest of the one color relations to show that

\igc{.7}{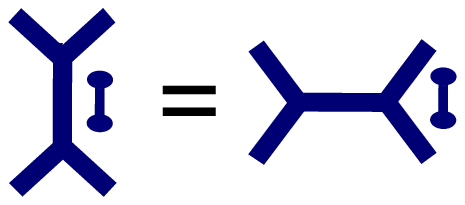}

Hence, (\ref{assoc1}) will hold, so long as $R$ acts freely on morphisms. Under this mild assumption, all the relations are required. While no
proof is presented here, it is safe to say that the category $\mc{SC}_1$ is universal amongst all categories
for which Rouquier complexes could be defined functorially up to Type I movie moves (under suitable conditions on color symmetry and torsion-free
double dot actions), and that these relations are effectively predetermined.

%
\subsection{Working over $\Z$}
\label{subsec-integers}
%

Knot theorists should be interested in a $\Z$-linear version of the Soergel bimodule story, because it could theoretically yield a functorial
link homology theory over $\Z$. We describe the $\Z$-linear version below. Because defining things over $\Z$ is not really the focus of this
paper, and because a thorough discussion would require poring over \cite{EKh} for coefficients, we do not provide rigorous proofs of the
statements in this section.

Ignoring the second equality in (\ref{dotslidenear}), which is equivalent to (\ref{dotslidesame}) after
multiplication by 2, every relation given has coefficients in $\Z$. One could use these relations to define a $\Z$-linear version of $\mc{SC}_1$
and $\mc{SC}_2$, and then use base extension to define the category over any other ring. The functor can easily be defined over $\Z$, as we have
demonstrated, and all the brute force checks work without resorting to other coefficients. Theorem \ref{mainthm} still holds for the $\Z$-linear
version of $\mc{SC}_2$.

In fact, the same method of proof (using homotopically isolated summands) will work over $\Z$ in most contexts. One begins by checking the
isomorphisms (\ref{dc-ii}) through (\ref{dc-ipi}). The only one which is in doubt is $B_i \TenR B_i \cong B_i\{-1\} \oplus B_i\{1\}$. So long as,
for each $i$, there is an adjacent color in $I$, we may use (\ref{iidecomp2}) to check this isomorphism. Otherwise, we are forced to use
(\ref{iidecomp}), which does not have integral coefficients.

For now, assume that adjacent colors are present; we will discuss the other case
below. One still has a map of algebras from $\mc{H}$ to the additive Grothendieck group of $\mc{SC}_1$. A close examination of the methods
used in the last chapter of \cite{EKh} will show that the graphical proofs which classify $\HOM(\emptyset,\ii)$ still work over $\Z$ in this
context. Boundary dots with a polynomial will be a spanning set for morphisms. One can still define a functor into a bimodule category to show
that this spanning set is in fact a basis. Therefore, the Hom space pairing on $\mc{SC}_1$ will induce a semi-linear pairing on $\mc{H}$, and it
will be the same pairing as before. Hom spaces will be free $\Z$-modules of the appropriate graded rank, and this knowledge suffices to use all
the homotopically isolated arguments.

\begin{remark} This statement does \emph{not} imply that $\mc{SC}$ will categorify the Hecke algebra when defined over $\Z$. There may be missing
idempotents, or extra non-isomorphic idempotents, so that the Grothendieck ring of the idempotent completion may be too big or small.
\end{remark}

If adjacent colors are not present, the easiest thing to do to prove Theorem \ref{mainthm} is to include $\mc{SC}_1(I)$ into a larger
$\mc{SC}_1(I^\prime)$ for which adjacent colors are present. Since this inclusion is faithful, all movie move checks which hold for $I^\prime$
will hold for $I$. Alternatively, one could use an extension of the category $\mc{SC}_1(I)$, extending the generating set by adding more
polynomials, either as originally done in \cite{EKh}, or by formally adding $\frac{1}{2}$ times the double dot. Both of these should give an
integral version of the category where the isomorphism (\ref{dc-ii}) holds, and where the graphical proofs of \cite{EKh} still work. Finally, if
one does not mind ignoring 2-torsion, defining the category over $\Z[\frac{1}{2}]$ will also work.


\vspace{5mm}
\hspace{-5mm}\emph{Ben Elias, Department of Mathematics, Columbia University, New York, NY 10027} \\
\vspace{3mm}
E-mail: belias@math.columbia.edu\\

\hspace{-5mm}\emph{Daniel Krasner, Department of Mathematics, Columbia University, New York, NY 10027} \\
\vspace{3mm}
E-mail: dkrasner@math.columbia.edu\\

\end{document}